\documentclass[a4paper,12pt,oneside]{article}
\usepackage[margin=1in]{geometry}
\usepackage[T1]{fontenc}
\usepackage{lmodern}
\usepackage{titlesec}
\usepackage{amsmath,amssymb}
\usepackage{xcolor,graphicx}
\usepackage{amsfonts}
\usepackage{amsthm}
\usepackage{amsopn}
\usepackage[caption=false]{subfig}
\usepackage{bbm}
\usepackage{hyperref}
\hypersetup{hidelinks}
\usepackage[capitalize,nameinlink]{cleveref}

\newcommand\innerp[2]{\langle #1, #2 \rangle}
\newcommand{\argmin}{\mathop{\mathrm{argmin}}} 
 
\newcommand*{\hermtrans}{^{\mathsf{T}}}
\newcommand{\metricmat}{\mathbf{R}_{\mathrm{U}}}

\newtheorem{theorem}{Theorem}[section]
\newtheorem{proposition}[theorem]{Proposition}
\newtheorem{corollary}[theorem]{Corollary}

\newtheorem{remark}{Remark}[section]
\newtheorem{definition}{Definition}[section]

\title{%
Dictionary-based model reduction for state estimation}
\author{%
  A. Nouy${}^{1}$,
  A. Pasco${}^{1}$}
\date{\medskip%
  \small %
  ${}^1$  Centrale Nantes, Nantes Université, \\ Laboratoire de Mathématiques Jean Leray UMR CNRS 6629\\
  \texttt{\{Anthony.Nouy,Alexandre.Pasco\}@ec-nantes.fr}
}

\begin{document}
\maketitle

\begin{abstract}
We consider the problem of state estimation from a few linear measurements, where the state to recover is an element of the manifold $\mathcal{M}$ of solutions of a parameter-dependent equation. 
The state is estimated using prior knowledge on $\mathcal{M}$ coming from model order reduction. 
Variational approaches based on linear approximation of $\mathcal{M}$, such as PBDW, yields a recovery error limited by the Kolmogorov width of $\mathcal{M}$. 
To overcome this issue, piecewise-affine approximations of $\mathcal{M}$ have also been considered, that consist in using a library of linear spaces among which one is selected by minimizing some distance to $\mathcal{M}$. 
In this paper, we propose a state estimation method relying on dictionary-based model reduction, where a space is selected from a library generated by a dictionary of snapshots, using a distance to the manifold. 
The selection is performed among a set of candidate spaces obtained from a set of  $\ell_1$-regularized least-squares problems. 
Then, in the framework of parameter-dependent operator equations (or PDEs) with affine parametrizations, we provide an efficient offline-online decomposition based on randomized linear algebra, that ensures efficient and stable computations while preserving theoretical guarantees.
\end{abstract}



\section{Introduction}
\label{sec:introduction}

This paper is concerned with state estimation (or inverse) problems that consist in the approximation (or recovery) of an element $u$ of a Hilbert space  $U$ using linear measurements and the additional knowledge that $u$ is the solution of a parameter-dependent equation \begin{equation}
\label{equ:general parametric problem}
    \mathcal{F}\big(u(\xi), \xi \big) = 0,
\end{equation}
for some unknown parameter $\xi$ in a parameter set $\mathcal{P}$. 
In other words, the sought state $u$ is an element of the so called solution manifold 
\begin{equation}
\label{equ:solution manifold}
    \mathcal{M}:= \{ u(\xi) : \xi \in\mathcal{P} \} \subset U,
\end{equation}
and we aim to approximate $u$ by using $m$ measurements $\ell_1(u)$, ..., $\ell_m(u)$, where the $\ell_i$ are continuous linear forms on $U$, along with some prior knowledge on $\mathcal{M}$ coming from model order reduction (MOR). 

A variational approach called parameterized-background data-weak (PBDW) was proposed in \cite{madayParameterizedbackgroundDataweakApproach2015}. 
It relies on the assumption that $\mathcal{M}$ can be well approximated by a low-dimensional linear space $V$, which can be obtained by, e.g., principal component analysis or greedy algorithms. 
The PBDW state estimate is then obtained as the sum of an element from the \emph{background space} $V$ and a correction  
in the \emph{observation space} $W$ generated by the Riesz representers of the linear forms $\ell_i$. 
However, the background space and the measurements should be compatible for ensuring stability of the estimation, especially in the case of noisy observations \cite{madayPBDWStateEstimation2015}. 
More precisely, 
stability is related to the alignment of spaces  $V$ and $W$.
There are two main approaches to overcome this issue. 
The first is to select carefully an observation space that is adapted to the given space $V$, e.g., by selecting the linear forms from a dictionary using a greedy algorithm \cite{binevGreedyAlgorithmsOptimal2018}. 
In practice, this approach may not perform well if the number and type of measurements (e.g., location of sensors) are restricted. 
The second approach is to consider different prior knowledge for $\mathcal{M}$, by constructing a space $V$ adapted to the measurements, e.g., using the optimization procedure from  \cite{cohenOptimalReducedModel2020a}, or 
by taking advantage of a wider knowledge coming from the construction of $V$ by a greedy algorithm, as proposed in \cite{binevDataAssimilationReduced2017,herzetPerformanceGuaranteesVariational2020}. 
In the case of noisy observations, the authors in \cite{taddeiAdaptiveParametrizedBackgroundDataWeak2017} propose a  regularized version of PBDW, that consists in considering elements from a bounded subset of the background space $V$.

The above approaches are all based on linear approximation, which may poorly perform in cases where the manifold $\mathcal{M}$ can not be well approximated by a single  linear space of small dimension, that is characterized by a slow decay of its Kolmogorov $n$-width 
\begin{equation}
\label{equ:kolmogorov m width}
    d_n(\mathcal{M})_U := 
    \inf_{\text{dim} V = n }
    \sup_{u \in\mathcal{M}} \min_{v \in V}
    \|u - v \|_U,
\end{equation}
which is the best worst-case error we can expect when approximating   $\mathcal{M}$ by a $n$-dimensional subspace $V$ in $U$. 

Hence, there are incentives to consider non-linear approximations of $\mathcal{M}$, like the piecewise-affine approximation proposed in \cite{cohenNonlinearReducedModels2022}, which relies on a library of  subspaces among which one space is selected by minimizing some  distance to the manifold, or the  approach based on manifold approximation \cite{cohenNonlinearApproximationSpaces2022a}.  

\subsection{Contributions and outline}

The outline is as follows.  
First in \Cref{sec:the onespace problem} we recall the original PBDW approach from \cite{madayParameterizedbackgroundDataweakApproach2015}, which was called the one-space problem in \cite{binevDataAssimilationReduced2017}. 
We also recall the error bounds for this approach, as well as its limitations.

Then in \Cref{sec:the multi-space problem} we describe a general multi-space problem, in which the background space $V$ is selected adaptively among a library $\mathcal{L}_n^N$ containing $N$ subspaces of dimension at most $n$. 
This is associated to another benchmark error, with faster decay than $d_n(\mathcal{M})_U$, which is the nonlinear Kolmogorov width from  
\cite{temlyakovNonlinearKolmogorovWidths1998} defined by
\begin{equation}
\label{equ:non linear kolmogorov n width}
    d_n(\mathcal{M}, N)_U := \inf_{\#\mathcal{L}_n^N=N} 
    \sup_{u\in\mathcal{M}} 
    \min_{V \in \mathcal{L}_n^N}   \min_{v \in V}
    \|u - v \|_U,
\end{equation}
where the infimum is taken over all libraries $\mathcal{L}_n^N$ of $N$ subspaces of dimension $n$.
The space is selected using a surrogate distance $\mathcal{S}$ to the manifold $\mathcal{M}$ and the associated recovery error is controlled, under certain assumptions, by the best recovery error among $\mathcal{L}_n^N$. 
This approach was introduced in \cite{cohenNonlinearReducedModels2022} for a specific choice of library. 
Here, it is presented in a general setting, but it is not a contribution per say. 

In \Cref{sec:randomized multi-space problem}, especially in \Cref{subsec:randomized selection criterion}, we present the first contribution of the present work, that is a randomized version $\mathcal{S}^{\Theta}$ of the selection method for the general multi-space problem, with a random linear mapping $\Theta$ based on randomized linear algebra (RLA) techniques for the estimation of $\mathcal{S}$. 
We prove that our randomized approach provides a priori error bounds similar to the non-randomized approach with high (user-defined) probability. 
This approach makes feasible an offline-online decomposition, with reasonable offline costs, low online costs and robustness to round-off errors.

The second and main contribution of our work is introduced in \Cref{sec:dictionary approach for inverse problem}, especially in \Cref{subsec:dictionary based multispace,subsec:offline-online decomposition}.
We propose a state estimation strategy using dictionary-based reduced order models \cite{balabanovRandomizedLinearAlgebra2021, kaulmannOnlineGreedyReduced2013}.
It consists in using a library of background low-dimensional spaces generated by a large dictionary $\mathcal{D}_K$ of $K$ snapshots in $\mathcal{M}$. 
Our approach is an adaptation of the approach from    
\cite{cohenNonlinearReducedModels2022} to a setting where we have access to a combinatorial number $\binom{K}{n}$ of candidate background spaces of low dimension $n$. 
With this dictionary-based approach, it is expected to  select a low-dimensional  background space among a huge candidate set, yielding a good approximation of  $\mathcal{M}$ and a good stability of the estimation problem in the case where only a few measurements are available.

We also provide for our approach an efficient offline-online decomposition in the framework of parameter-dependent operator equations (or PDEs) with affine parametrization. 
Indeed, classical offline computations for residual-based quantities often lead to prohibitive offline costs when it comes to dictionary-based approaches, as pointed out in \cite{kaulmannOnlineGreedyReduced2013, balabanovRandomizedLinearAlgebra2021}.
They also tend to be sensible to round-off errors. 
We solve this issue by using the randomized quantity $\mathcal{S}^{\Theta}$  that we proposed in \Cref{subsec:randomized selection criterion}. 

Finally in \Cref{sec:numerical examples} we test our approach on two different numerical examples. 
First a classical thermal-block diffusion problem, then an advection-diffusion problem.

\subsection{Setting and notations}
Instead of referring to the measurements $\ell_1(u)$, ..., $ \ell_m(u)$, we will equivalently  consider the orthogonal projection $w = P_W u $ of the state $u$ onto the observation space 
\begin{equation}
\label{equ:observation space def}
    W := \text{span} \{ R_U^{-1} \ell_i: 1\leq i\leq m \} \subset U,
\end{equation}
where $R_U : U \rightarrow U'$ is the Riesz map, which is the unique linear isometry from $U$ to its dual $U'$, from the Riesz Representation Theorem.
We assume (w.l.o.g.) that the linear forms are linearly independent, which implies that the space $W$ is of dimension $m$. 
The problem of state estimation is equivalent  to finding a recovery map $A : W \to U$ which associates to an observation $w \in W$ an element  $A(w) \in U$. 

Our methodology is valid for real or complex Hilbert spaces, but for the sake of simplicity, we restrict the presentation to the real case. 
In practice, although most of our results are valid for a general Hilbert setting, the space $U$ is of finite  dimension 
\[
\mathcal{N} := \text{dim}(U).
\]
In the context of PDEs, the space arises from some discretization (e.g., based on finite elements, finite volumes...) and its dimension $\mathcal{N}$ is usually very large to ensure high fidelity towards the true solution. 
The space $U$ and its dual $U'$ are then identified with  $\mathbb{R}^{\mathcal{N}}$, respectively endowed with inner products 
$\langle\cdot, \cdot\rangle_U = \langle\cdot, \metricmat \cdot \rangle$ and $\langle\cdot, \cdot\rangle_{U'} = \langle \cdot, \metricmat^{-1} \cdot \rangle$, where $\langle \cdot, \cdot \rangle$ denotes the canonical $\ell_2$-inner product in $\mathbb{R}^{\mathcal{N}}$, and where $\metricmat\in\mathbb{R}^{\mathcal{N}\times \mathcal{N}}$ is the positive definite symmetric  matrix associated with the Riesz map $R_U$.
Discrete representations of operators and vectors will be systematically written with bold notations.

\section{The one-space problem, or PBDW}
\label{sec:the onespace problem}

The one-space problem was initially called Parameterized-Background Data Weak (PBDW) and formulated in \cite{madayParameterizedbackgroundDataweakApproach2015}. 
We denote by $A_V : W \rightarrow U$ the corresponding recovery map for some background space $V$. 
In this section, we consider that the only knowledge we have on $\mathcal{M}$ is that it is well approximated by a $n$-dimensional space $V_n$, with
\begin{equation*}
    \text{dist}(V_n, \mathcal{M}) = \sup_{u\in \mathcal{M}} \min_{v\in V_n} \Vert u - v \Vert \leq \varepsilon_n.
\end{equation*}
Note that for the one-space problem, the precision $\varepsilon_n$ is not required to be known to compute the state estimate, while it is the case in the nested multi-space problem in \Cref{subsec:nested multi-space}.

\subsection{Variational formulation}
\label{subsec:variational formulation}

The PBDW problem was first formulated as a variational problem. 
For an observation $w = P_{W} u$, the recovery $A_{V_n}(w)$ is obtained by selecting the background term $v^*\in V_n$ which requires the smallest update $\eta^*\in W$. 
In other words,
\begin{align}
\label{equ:pbdw ls formulation}
 A_{V_n}(w) := v^* + \eta^* \quad \text{with} \quad \left\{\begin{aligned}
    &v^* := \arg\min_{v\in V_n} \|P_{W} (u - v)\|_U \\
     &\eta^* := P_{W} (u - v^*)
    \end{aligned}\right..
\end{align}
In a finite dimensional setting, with $U = \mathbb{R}^{\mathcal{N}}$, the problem \eqref{equ:pbdw ls formulation} has the following  algebraic form:
\begin{align}
\label{equ:pbdw ls formulation algebraic}
A_{V_n}(w) := v^* + \eta^* \quad \text{with} \quad \left\{
\begin{aligned}
    & v^* = \mathbf{V} \mathbf{v}^*, \quad \mathbf{v}^* := \arg\min_{\mathbf{v} \in \mathbb{R}^n} 
    \|\mathbf{C} \mathbf{v} - \mathbf{w}\|_{2},& \\
    & \eta^* =\mathbf{W} \boldsymbol{\eta}^*, \quad \boldsymbol{\eta}^* := \mathbf{w} - \mathbf{C} \mathbf{v}^*,&
\end{aligned}
\right.
\end{align}
where $\mathbf{V} \in \mathbb{R}^{\mathcal{N} \times n}$ and $\mathbf{W} \in \mathbb{R}^{\mathcal{N} \times m}$ are the matrices whose columns form an orthonormal basis of $V_n$ and $W$ respectively, meaning that $\mathbf{V}\hermtrans \metricmat \mathbf{V} = \mathbf{I}_n$ and $\mathbf{W}\hermtrans \metricmat \mathbf{W} = \mathbf{I}_m$.
The matrix $\mathbf{C = W\hermtrans \metricmat V} \in \mathbb{R}^{m\times n}$ is the cross gramian matrix, and $\mathbf w \in \mathbb{R}^{m}$ is such that $w = \mathbf{W w}$. 
Problem \eqref{equ:pbdw ls formulation} is well-posed if and only if $V_n ~\cap~ W^{\perp} = \{0\}$. 
This is equivalent to 
$\sigma_{\text{min}}(\mathbf{C})>0$ with $n\leq m$. 
For the PBDW estimation, we assume that this condition is satisfied.
The analysis from \cite{binevDataAssimilationReduced2017} then ensures the error bound
\begin{equation}
\label{equ:pbdw error bound}
    \|u - A_{V_n}(w)\|_U 
    \leq \mu(V_n, W) \text{dist}(V_n + W, \mathcal{M})
    \leq \mu(V_n, W) \varepsilon_n,
\end{equation}
where $V_n+W$ is the set spanned by all sums of elements of $V_n$ and $W$, and where $\mu(V_n, W)$ is the inverse of the smallest singular value $\beta(V_n, W)$ of the projection operator $P_{W}$ restricted to $V_n$. 
These constants are defined by
\begin{align}
\label{equ:pbdw infsup constants}
\begin{aligned}
    \mu(V_n, W) &:= \sup_{\eta \in W^{\perp}} \frac{\|\eta\|_U}{\|\eta - P_{V_n} \eta\|_U} 
    = \sigma_{\text{min}}(\mathbf C)^{-1}, \\
    \beta(V_n, W) &:= \inf_{v \in V_n} \frac{\|P_{W} v\|_U}{\|v\|_U} 
    = \inf_{v \in V_n} \sup_{w \in W} \frac{\langle v, w \rangle_U}{\|v\|_U \|w\|_U} 
    = \sigma_{\text{min}}(\mathbf C).
\end{aligned}
\end{align}

\subsection{Geometric formulation and optimality}
\label{subsec:Geometric formulation and optimality}

The geometric interpretation of the problem described in \Cref{subsec:variational formulation} was provided and analyzed in \cite{binevDataAssimilationReduced2017}. 
In the one-space problem, the manifold can be considered as the  cylinder $\mathcal{K}^n$ defined by
\begin{equation}
\label{equ:definition cylinder around manifold}
    \mathcal{K}^n := \{ v \in U : \text{dist}(v, V_n) \leq \varepsilon_n \}.
\end{equation}
Given an observation $w = P_W u $, all the knowledge we have about $u$ is that $u \in \mathcal{K}^n_w$, where 
\begin{equation}
\label{equ:observation ellipsoid}
    \mathcal{K}^n_w := \mathcal{K}^n \cap  U_w, \quad 
    U_w := w + W^{\perp}.
\end{equation}
The optimal recovery map, in the sense of the worst case error in $\mathcal{K}^n$, is the one mapping $w$ to the Chebyshev center of $\mathcal{K}^n_w$, where the Chebyshev center is defined for any bounded set $\mathcal{X} \subset U$ by
\begin{equation}
\label{equ:Chebyshev center}
    \text{cen}(\mathcal X) := 
    \argmin_{v \in U} r_{\mathcal{X}}(v)
    \quad \text{where} \quad 
    r_{\mathcal{X}}(v) := 
    \min \{ r': \mathcal X \subset \overline{\mathcal{B}(v, r')} \}.
\end{equation}
Note that \cite[Lemma 2.1]{binevDataAssimilationReduced2017} ensures that both $r_{\mathcal{X}}$ and cen$(\mathcal{X})$ are well defined.
It has been shown in \cite[Theorem 2.8]{binevDataAssimilationReduced2017} that $A_{V_n}$ is the optimal map when the manifold is the cylinder $\mathcal{K}^n$. 
In other words $A_{V_n}(w) = \text{cen}(\mathcal{K}^n_{w})$ and
\begin{equation}
    A_{V_n} = \argmin_{A:W \rightarrow U}
    \sup_{u \in \mathcal{K}^n} \|u - A(P_W u)\|_U,
\end{equation}
where the minimum is taken over all mappings from $W$ to $U$, including non-linear ones. 
In \Cref{fig:pbdw geometric} we illustrate  this geometric formulation.

\begin{figure}[!ht]
    \centering
    \subfloat{\label{fig:a}\includegraphics[page=1, width=0.4\textwidth]{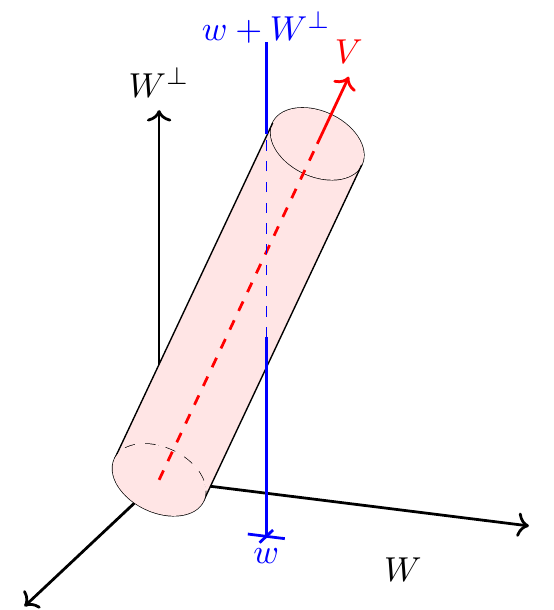}}
    \hspace{2cm}
    \subfloat{\label{fig:b}\includegraphics[page=2, width=0.21\textwidth]{figures.pdf}}
    \caption[PBDW geometric]{Geometric interpretation of the PBDW recovery for one observation $w=P_W u$, with dim$(W)=2$ and dim$(V)=1$.}
    \label{fig:pbdw geometric}
\end{figure}

The one-space problem  may have bad performances when the background space $V_n$ is not well aligned with the observation space $W$, which leads to a high $\mu(V_n, W)$, or when the manifold cannot be well approximated by a single linear space, which leads to a high $\varepsilon_n$. 
The approximation by a single subspace is inherently limited by the Kolmogorov $m$-width via the bound $\varepsilon_n \geq d_m(\mathcal{M})_U$ because $n \leq m$.

\section{The multi-space problem}
\label{sec:the multi-space problem}

Several approaches have been proposed for addressing the issues of the one-space approach \cite{cohenOptimalReducedModel2020a, herzetPerformanceGuaranteesVariational2020,cohenNonlinearReducedModels2022}.
These are \emph{multi-space} approaches that consist in exploiting a library of $N$ spaces $V_1$, ..., $V_N$.

First in \Cref{subsec:nested multi-space} we present the multi-space approach as first described in \cite{binevDataAssimilationReduced2017}. 
Then in \Cref{subsec:general approach} we present the multi-space approach proposed in \cite{cohenNonlinearReducedModels2022}, with a selection method based on some surrogate distance to the manifold, and recall the associated oracle error bound which holds not only in the piecewise-affine framework considered in \cite{cohenNonlinearReducedModels2022}. 
This piecewise affine framework is considered in \Cref{subsec:piecewise multi-space}. 
Finally in \Cref{subsec:parametric pdes}, we focus on the case where problem \eqref{equ:general parametric problem} is a parameter-dependent PDE, which provides (under suitable conditions) a natural residual-based surrogate distance.

\subsection{Nested multi-space}
\label{subsec:nested multi-space}
The library corresponding to the multi-space approach from \cite{binevDataAssimilationReduced2017} is constituted of a nested sequence of subspaces with increasing approximation powers. 
In this subsection, we consider a library composed of $N = n \leq m$ nested subspaces such that 
\begin{equation*}
    V_1 \subset \cdots \subset V_n,
    \quad \text{dim}(V_j) = j,
    \quad \text{dist}(\mathcal{M}, V_j) \leq \varepsilon_j,
    \quad 1\leq j\leq n.
\end{equation*}
For an observation $w\in W$ we consider the compact set
\begin{equation}
\label{equ:multi-space observed ellipsoid}
    \mathcal{K}^{\text{mult}}_w := U_w \cap \mathcal{K}^{\text{mult}},
    \quad \mathcal{K}^{\text{mult}} := \bigcap_{j=1}^n \mathcal{K}^j, \quad U_w := w + W^{\perp}.
\end{equation}
where the cylinders $\mathcal{K}^j$ for $1\leq j\leq n$ are defined by \eqref{equ:definition cylinder around manifold}, as for the one-space problem. The optimal recovery map for the worst case error would be the map $w \mapsto \text{cen}(\mathcal{K}^{\text{mult}}_w)$.
However, contrary to the one-space problem, this optimal map is not computable in general. In \cite{binevDataAssimilationReduced2017}, the authors propose to select any map $A^{\text{mult}}$ such that 
\begin{equation}
\label{equ:nested multi-space map 1}
    A^{\text{mult}}(w) \in \mathcal{K}^{\text{mult}}_w.
\end{equation}
They showed that any map satisfying \eqref{equ:nested multi-space map 1} satisfies the oracle error bound
\begin{equation}
\label{equ:nested multi-space map 1 error bound}
    \sup_{u\in\mathcal K^{\text{mult}}} \|u - A^{\text{mult}}(P_W u)\| \leq 
    \sqrt{2} \min_{1\leq j \leq n} \mu(V_j, W) \varepsilon_j.
\end{equation}
They also propose numerical algorithms to compute such a map. 
It is important to note that the proposed algorithms require to know the widths $(\varepsilon_j)_{1\leq j\leq n}$ to be performed, whereas it is not the case for the one-space problem. 
This approach allows to take advantage of the best compromise possible between $\varepsilon_j$ and $\mu(V_j, W)$ among all the available spaces. 
However it is not well adapted to problems with slow decay of Kolmogorov $n$-width, similarly as the one-space problem.

We end this subsection by citing the works from \cite{herzetPerformanceGuaranteesVariational2020} which also consider the nested multi-space framework. 
The authors propose a recovery map $A$ such that $A(w) \in V_n \cap \mathcal{K}^{\text{mult}}$ minimizes the update $\|w - P_W A(w)\|_U$. 
Note that \cite{herzetPerformanceGuaranteesVariational2020} also extends the multi-space problem to the case $m<n$ and provides error bounds, ensuring stability when increasing the knowledge with $n-m$ new spaces with better approximation power, although the error bounds provided are not sharp. 
We also cite the works from \cite{cohenOptimalReducedModel2020a} in which the background space is obtained via a convex optimization problem.

\subsection{General approach}
\label{subsec:general approach}

The general multi-space approach consists in using not a single space $V_n$ as background, as for the initial PBDW approach, but rather a library 
\begin{equation*}
    \mathcal{L}_n^N := \{V_1, \cdots, V_N\}
\end{equation*}
of $N$ subspaces of dimension at most $n$. We suppose that
\begin{equation*}
    n_k := \text{dim}(V_k) \leq n \leq m,
    \hspace{5mm} 1\leq k\leq N.
\end{equation*}
The benchmark error for this non-linear model order reduction method is the library-based non linear Kolmogorov $n$-width from \cite{temlyakovNonlinearKolmogorovWidths1998} defined by \eqref{equ:non linear kolmogorov n width}.
For some $N=N(n)$, the library width $d_n(\mathcal{M}, N(n))_U$ may have a (much) faster decay with $n$  than $d_n(\mathcal{M})_U$, so that a small error may be obtained with spaces of low dimension $n$, which can be crucial for our state estimation problem. 
To each subspace is associated the one-space recovery map
\begin{equation}
\label{equ:one map in multi-space}
    A_k := A_{V_k}.
\end{equation}
For a given observation $w$, it was proposed in \cite{cohenNonlinearReducedModels2022} to select one of those maps by minimizing some surrogate distance to the manifold $\mathcal{S}(\cdot, \mathcal{M})$. 
The idea is to select $k^* = k^*_{\mathcal{S}}(w)$ whose associated one-space recovery is the closest to the manifold, in the sense that it minimizes $\mathcal{S}(\cdot, \mathcal{M})$. 
This recovery is then denoted as $A^{\text{mult}}_{\mathcal{S}}(w)$. 
In other words,
\begin{equation}
\label{equ:multi-space space selection}
    A^{\text{mult}}_{\mathcal{S}}(w) := A_{k^*}(w)
    \quad\text{with} \quad 
    k^* \in \argmin_{1\leq k\leq N} ~ \mathcal{S}(A_k (w), \mathcal{M}).
\end{equation}
The best choice for $\mathcal{S}$ would be the distance based on the $\|\cdot \|_U$ norm. 
However, this approach may not be computationally feasible in most practical cases where $\mathcal{N}$ is large and elements in $\mathcal{M}$ are expensive to compute or store. 
Still, assuming that $\mathcal{S}(\cdot, \mathcal{M})$ controls the true distance to $\mathcal{M}$, interesting bounds can be shown. 
Let us assume that there exist constants $0< c \leq C <  + \infty$ such that
\begin{equation}
\label{equ:distance controled by surrogate}
    c ~ \text{dist}(v,\mathcal{M}) 
    \leq \mathcal{S}(v,\mathcal{M})
    \leq C ~ \text{dist}(v,\mathcal{M}).
\end{equation}
Then in \cite[Theorem 3.4]{cohenNonlinearReducedModels2022}, the authors manage to control the error of the selected estimate by the best recovery error among the library. 
Although they work with a library built by a piecewise affine model reduction approach, their result still holds for a general library. 
They introduce and analyse the constant $\mu(\mathcal{M}, W)$ defined by
\begin{subequations}\label{equ:general stability constant}
\begin{equation}
    \mu(\mathcal{M}, W) := \frac{1}{2} \sup_{\sigma >0} 
    \frac{\delta_{\sigma} - \delta_0}{\sigma}
    \end{equation}
with 
\begin{equation}
    \delta_{\sigma} := \max_{w\in W} \text{diam}(\mathcal{M}_{\sigma}\cap U_w)
    \quad \text{and} \quad 
    \mathcal{M}_{\sigma} := \big\{
    v \in U : \text{dist}(v, \mathcal{M}) \leq \sigma
    \big\},
\end{equation}
\end{subequations}
where $\mathcal{M}_\sigma$ is the $\sigma$-fattening of $\mathcal{M}$. 
The constant $\mu(\mathcal{M}, W)$ reflects the stability of the recovery problem, independently on the method used for the approximation of $\mathcal{M}$.
In other words, it reflects how well $\mathcal{M}$ and $W$ are aligned. 
The preferable case is when $\delta_0=0$.
It is equivalent to $P_W$ being injective on $\mathcal{M}$. 
In this case, the error on $A_{\mathcal{S}}^{\text{mult}}$ is truly controlled by the best possible error within the library, as shown in \Cref{prop:general oracle bound}.

\begin{proposition}
\label{prop:general oracle bound}
Assume that $\mu(\mathcal{M}, W) < \infty$, then for all $u\in\mathcal{M}$,
\begin{equation}
\label{equ:general oracle bound}
    \| u - A^{\text{mult}}_{\mathcal{S}}(P_W u) \|_U 
    \leq \delta_0 + 2 \kappa \mu(\mathcal{M}, W) 
    \min_{1\leq k \leq N} \|u - A_{k}(P_W u) \|_U,
\end{equation}
with $\kappa := \frac{C}{c}$ and $c,C$ are the constants from \eqref{equ:distance controled by surrogate}.
\end{proposition}
\begin{proof}
    Let us pick the $k_{\text{best}}$-th subspace from $\mathcal{L}_n^N$ as the space associated to the best one-space recovery,
    \[
    k_{\text{best}}=\argmin_{1 \leq k \leq N} \| u - A_k(w) \|_U.
    \]
    Using \eqref{equ:distance controled by surrogate} and the definition of $k^*$ from \eqref{equ:multi-space space selection}, it follows that
    \[
    \text{dist}(A_{k^*}(w), \mathcal{M}) 
    \leq \frac{1}{c} \mathcal{S}(A_{k^*}(w), \mathcal{M}) 
    \leq \frac{1}{c} \mathcal{S}(A_{k_{\text{best}}}(w), \mathcal{M}) 
    \leq \frac{C}{c} \text{dist}(A_{k_{\text{best}}}(w), \mathcal{M}), 
    \]
    that is 
    \[\text{dist}(A_{k^*}(w), \mathcal{M}) \leq \kappa \lambda,
    \quad \text{with}\quad  
    \lambda := \text{dist}(A_{k_{\text{best}}}(w), \mathcal{M}).
    \]
    Hence, since $A_{k^*}(w) \in U_w$ by definition \eqref{equ:pbdw ls formulation},
    $A_{k^*}(w) \in \mathcal{M}_{\kappa \lambda} \cap U_w$. Since $u$ also lies in this set, the recovery error using $A^{\textup{mult}}_{\mathcal{S}} = A_{k^*}$ is bounded via
    \[
    \| u - A^{\textup{mult}}_{\mathcal{S}}(w) \|_U 
    \leq \text{diam}(\mathcal{M}_{\kappa \lambda} \cap U_w)
    \leq \delta_{\kappa \lambda}.
    \]
    We now distinguish the cases $\lambda=0$ and $\lambda>0$. 
    If $\lambda=0$, we have $\| u - A^{\textup{mult}}_{\mathcal{S}}(w) \|_U \leq \delta_0$, which   implies \eqref{equ:general oracle bound}. 
    If $\lambda >0$,  using  
    \[
    \delta_{\kappa \lambda}
    \leq \delta_0 + 2\kappa \lambda \mu(\mathcal{M}, W),
    \]
    and $\lambda \leq \|u - A_{k_{\text{best}}}(w) \|_U$ (since $u\in \mathcal{M}$), we also obtain  \eqref{equ:general oracle bound}. 
\end{proof}

This general library-based approach allows to perform one-space recovery in background spaces of low dimension, hence potentially good $\mu(V_{k^*}, W)$, while ensuring good approximation power of the selected background space. 
However, as stated in \cite{balabanovRandomizedLinearAlgebra2021}, some problems require a very large number $N$ of spaces to ensure good precisions $(\varepsilon_k)_{1\leq k\leq N}$. 
Hence testing each of the recovery maps $A_k$ for every new observation $w\in W$ can be a computational burden online.
Our dictionary-based approach described in \Cref{sec:dictionary approach for inverse problem}, as well as  our random sketching approach in \Cref{subsec:randomized selection criterion}, aim to circumvent this issue.

\subsection{Piecewise multi-space}
\label{subsec:piecewise multi-space}

This subsection is devoted to the case where the spaces in the library are obtained via partitioning of the parameter space $\mathcal{P}$ and the manifold $\mathcal{M}$, such that
\begin{equation}
\label{equ:partition of P and M}
    \mathcal{P} = \bigcup_{k=1}^N \mathcal{P}^{(k)},
    \hspace{5mm}
    \mathcal{M} = \bigcup_{k=1}^N \mathcal{M}^{(k)},
    \hspace{5mm} 
    \mathcal{M}^{(k)} = \{u(\xi):\xi\in\mathcal{P}^{(k)}\}.
\end{equation}
The piecewise multi-space approach  then consists in building $N$  spaces with moderate dimension, where the space $V_k$ approximates the piece $\mathcal{M}^{(k)}$.

Such nonlinear, local model reduction approach has already been considered in various works, with mainly two different approaches.
The first one is local in the parameter space, as considered in \cite{drohmannAdaptiveReducedBasis2011,eftangHpCertifiedReduced2010,eftangHpCertifiedReduced2011,haasdonkTrainingSetMultiple2011a}. 
It consists in building the partition of $\mathcal{P}$ using a refinement method, then approximating each local piece $\mathcal{M}^{(k)}$.
This corresponds to the framework considered in \cite{cohenNonlinearReducedModels2022} for state estimation.
The second one is local in the state space, as considered in \cite{amsallemNonlinearModelOrder2012,washabaughNonlinearModelReduction2012,peherstorferLocalizedDiscreteEmpirical2014,amsallemFastLocalReduced2015,amsallemPEBLROMProjectionerrorBased2016}, which consists in building directly the local pieces $(\mathcal{M}^{(k)})_{1\leq k\leq N}$, without explicitly partitioning $\mathcal{P}$, using clustering methods such as $K$-means.

Recovery algorithms using such reduced spaces were considered in \cite{cohenNonlinearReducedModels2022}, and it was the initial framework for the general approach described in \Cref{subsec:general approach}. 
They also covered the case of a library of affine spaces, but we do not elaborate on it in the present paper. 
In this subsection, we consider a library
\begin{equation*}
    \mathcal{L}_n^N = \{V_k : 1\leq k \leq N \}
    \quad  \text{with} \quad 
    \text{dim}(V_k) = n_k \leq n \leq m,
\end{equation*}
with the one-space constants associated to each sub-manifold,
\begin{equation*}
\begin{gathered}
    \varepsilon_k := \text{dist}( \mathcal{M}^{(k)}, V_k ),
    \quad \mu_k := \mu(V_k, W),
    \quad 1\leq k\leq N.
\end{gathered}
\end{equation*}
In this framework, using the general approach described in \Cref{subsec:general approach}, it is possible to ensure error bounds depending on the quantity $\mathcal{S}$ used for model selection. 
If $\mathcal{S}$ satisfies equation \eqref{equ:distance controled by surrogate}, then for $\sigma := \max_{1\leq k\leq N} \varepsilon_k \mu_k$ we have the error bound \cite[Theorem 3.2]{cohenNonlinearReducedModels2022}
\begin{equation}
\label{equ:piecewise affine error bound}
    \|u - A^{\text{mult}}_{\mathcal{S}}(P_W u) \|_U \leq
    \delta_{\kappa\sigma},
\end{equation}
with $\delta_{\kappa\sigma}$ defined in \eqref{equ:general stability constant}. 
This piecewise multi-space approach may suffer from computational problems when the required number $N$ of subspaces becomes too large to satisfy a desired precision. 

For example $N$ may become too large when we generate the local pieces by partitioning the parameter space as in \cite{amsallemNonlinearModelOrder2012,peherstorferLocalizedDiscreteEmpirical2014,amsallemPEBLROMProjectionerrorBased2016,cohenNonlinearReducedModels2022}.
As pointed out in \cite{peherstorferLocalizedDiscreteEmpirical2014,amsallemPEBLROMProjectionerrorBased2016}, this can be caused by a poor iterative partitioning of $\mathcal{P}$, or by a troublesome property of the mapping $\xi \mapsto u(\xi)$ such as periodicity.
A way to circumvent this is using a local approach in the state space as in \cite{amsallemNonlinearModelOrder2012,washabaughNonlinearModelReduction2012,peherstorferLocalizedDiscreteEmpirical2014,amsallemFastLocalReduced2015,amsallemPEBLROMProjectionerrorBased2016}, although to the best of our knowledge, this has not yet been investigated in the inverse problem setting we consider here.

Another example of $N$ becoming too large is when the library-width $d_n(\mathcal{M},N)_U$, defined by \eqref{equ:non linear kolmogorov n width}, has a slow decay with respect to $N$ for fixed $n$.

Both of those issues may be circumvented by the dictionary-based approach proposed in \Cref{sec:dictionary approach for inverse problem}, since it is able to generate a wide variety of subspaces, independently on the parametrization.

\subsection{Parameter-dependent operator equations}
\label{subsec:parametric pdes}

In this subsection we assume that the parametric problem \eqref{equ:general parametric problem} is actually a parameter-dependent  
operator equation (or PDE)
\begin{equation}
\label{equ:parametric pde}
    B(\xi) u(\xi) = f(\xi)
\end{equation}
where $B(\xi): U \to U'$ is the operator and $f(\xi) \in U'$ the right-hand side. In this framework, it is common to consider a residual-based error as a surrogate to the error in the norm $\|\cdot\|_U$. We define for all $v\in U$,
\begin{subequations}\label{equ:S2 and residual norm definition}
\begin{equation}
    \mathcal{S} (v, \mathcal{M}) = \min_{\xi \in \mathcal{P}} \mathcal{R}(v , \xi),
    \quad
    \end{equation}
    with
\begin{equation}    
    \mathcal{R}(v ,  \xi) := \| r(v , \xi) \|_{U'} , \quad    
    r(v,\xi) := B(\xi) v - f(\xi) \in U', 
    \quad\forall \xi\in\mathcal{P}.
\end{equation}
\end{subequations}
For the quantity $\mathcal{S}$ to be equivalent to the error $\| u -v \|_U$, and in order to derive error bounds, some assumptions are required on the operator $B(\xi)$. 
In the case where $B(\xi)$ is linear, we assume that its singular values are uniformly bounded in $\xi$, that is 
\begin{equation}
\label{equ:fom infsup constants}
    0 < c \leq \min_{v\in U} \frac{\| B(\xi) v \|_{U'} }{\|v\|_U}
    \leq \max_{v \in U} \frac{\| B(\xi) v \|_{U'} }{\|v\|_U}
    \leq C < \infty,
\end{equation}
for all $\xi \in \mathcal{P}$, for some constants $c$ and $C$. This implies that 
 for all $v \in U$,
\begin{subequations}
\label{equ:residual distance bounds}
\begin{equation}
    c \| v - u(\xi) \|_U 
    \leq \mathcal{R}(v,\xi) 
     \leq  C  \|v - u(\xi)\|_U, 
    \quad  \forall \xi\in\mathcal{P}, 
    \end{equation}
    and
 \begin{equation}   
    c~ \text{dist}(v,\mathcal{M}) 
    \leq  \mathcal{S}(v,\mathcal{M})
    \leq  C~  \text{dist}(v,\mathcal{M}).
\end{equation}
\end{subequations}
Moreover, under additional assumptions on $B$ and $f$, it is possible to compute $\mathcal{R}(v,\cdot)$ in an online efficient way, i.e. with costs independent on $\mathcal{N}$, by taking advantage of the fact that $v$ lies in a low-dimensional subspace. 
Furthermore, computing $\mathcal{S}$ does not require us to know any element in $\mathcal{M}$, whereas the true distance does. 
We provide a more detailed discussion on the computational aspects in \Cref{subsec:offline-online decomposition}. 
\begin{remark}
    In \cite{cohenNonlinearReducedModels2022} the authors proposed another residual-based nonlinear recovery map, which is
    \[
       A(w)  = \argmin_{v\in w + W^{\perp}} \min_{\xi\in\mathcal{P}} \mathcal{R}(v,\xi),
    \]
    which yields an error $\Vert A(w) - u \Vert_U \le \delta_0$ in our noiseless setting. 
    In some cases, as when $B(\xi)$ is linear and $\xi \mapsto B(\xi)$ and $ \xi \mapsto f(\xi)$ are affine functions of $\xi$, the function $(v,\xi)\mapsto \mathcal{R}(v,\xi)^2$ is convex and quadratic in each component, hence the previous problem can be tackled with an alternating minimization procedure. 
    Numerical experiments from \cite{cohenNonlinearReducedModels2022} have shown very good results for large enough number of measurements $m$. 
    However, it requires solving online about $m$ high dimensional problems, which is prohibitive for large $\mathcal{N}$.
\end{remark}
\begin{remark}
The properties \eqref{equ:residual distance bounds} also hold in the case of a  Lipschitz continuous and strongly monotone nonlinear operator $B(\xi)$, where $c$ is a uniform lower bound of the strong monotonicity constant and $C$ is a uniform bound of the Lipschitz constant. 
However, in a nonlinear setting, additional efforts are required for obtaining efficient online and offline computations. 
\end{remark}

\begin{remark}
    A preconditioned residual could be used to obtain a better condition number $\kappa$. 
    Using preconditioners to better estimate the $U$-norm in a model order reduction setting  was proposed in \cite{zahmInterpolationInverseOperators2016, balabanovPreconditionersModelOrder2021}.
\end{remark}

\section{Randomized multi-space approach}
\label{sec:randomized multi-space problem}

First in \Cref{subsec:randomized linear algebra} we recall some required preliminaries on randomized linear algebra. 
Then in \Cref{subsec:randomized selection criterion}, we propose to use a sketched (or randomized) version of the selection criterion from \Cref{subsec:parametric pdes} in the framework of parameter-dependent operator equations. 
We show that, with a rather small sketch and high probability, an oracle bound similar to the non-sketched approach holds.

\subsection{Randomized linear algebra}
\label{subsec:randomized linear algebra}

A basic task of RLA is to approximate norms and inner products of high dimensional vectors, by 
operating on lower dimensional vectors obtained through a \emph{random embedding} or \emph{random sketching}.  
Such embedding does not only reduce the computational time for evaluating norms and inner products, but also reduce the memory consumption since only the embedded vectors need to be stored. 

We first recall the notion of $U \rightarrow \ell_2$ embedding for a subspace $Y \subset U$, which is a generalization from \cite{balabanovRandomizedLinearAlgebra2019} of the notion of $\ell_2$ embedding for a subspace introduced in \cite{woodruffComputationalAdvertisingTechniques2014}. 
Let $\mathbf{\Theta} \in \mathbb{R}^{k \times \mathcal{\mathcal{N}}}$ be a random matrix, where $k$ is the embedding dimension, that represents a linear map $\Theta$ from $(U,\langle\cdot,\cdot\rangle_U)$ to $(\mathbb{R}^{k},\langle\cdot,\cdot\rangle)$. 
This map allows to define (sketched) semi-inner products
\begin{equation}
\label{equ:rla semi inner products}
    \langle \cdot, \cdot \rangle_U^{\Theta} = 
    \langle \mathbf{\Theta} \cdot, \mathbf{\Theta} \cdot \rangle
    \hspace{5mm}\text{and}\hspace{5mm}
    \langle \cdot, \cdot \rangle_{U'}^{\Theta} = 
    \langle \mathbf{\Theta} \metricmat^{-1} \cdot, \mathbf{\Theta} \metricmat^{-1} \cdot \rangle, 
\end{equation}
and associated   semi-norms $\|\cdot\|_U^{\Theta}$ and $\|\cdot\|_{U'}^{\Theta}$ respectively. 
Recall that $\metricmat$ is the inner product matrix in $U$ such that $\innerp{\cdot}{\cdot}_U = \innerp{\cdot}{\metricmat\cdot}$.
We now define formally the notion of subspace embedding. 
\begin{definition}
\label{def:subspace embedding def}
$\Theta$ is called a $U\rightarrow \ell_2$ $\epsilon$-subspace embedding for $Y$, with  $\epsilon\in [0,1)$, if 
\begin{equation}
\label{equ:subspace embedding def}
 \Big\vert \|v\|_U ^2 - \|v\|_U^{\Theta~2} \Big\vert 
 \leq \epsilon \|v\|_U ^2, \quad \forall v \in Y.
\end{equation}
\end{definition}
It can be shown that if $\Theta$ is a $U\rightarrow \ell_2$ $\epsilon$-subspace embedding for $Y$ then $\Theta R_U^{-1}$ is a $U'\rightarrow \ell_2$ $\epsilon$-subspace embedding for $Y'$. 
Hence, both $\| \cdot \|_U$ and $\|\cdot\|_{U'}$ are well approximated by their sketched versions. 

This definition is a property of $\Theta$ for a specific subspace $Y$. 
We are interested in the probability of failure for the random embedding $\Theta$ to be a $U\rightarrow \ell_2$ $\epsilon$-subspace embedding for any $d$-dimensional subspace $Y$ of $U$. With this in mind, the notion of oblivious embedding is defined.
\begin{definition}
\label{def:oublivous embedding def}
$\Theta$ is called a $(\epsilon, \delta, d)$ oblivious $U\rightarrow \ell_2$ subspace embedding if for any $d$-dimensional subspace $Y$ of $U$,
\begin{equation}
\label{equ:oublivous embedding def}
    \mathbb{P}(\Theta \text{ is a } U \rightarrow \ell_2 ~\epsilon\text{-subspace embedding for } Y) 
    \geq 1-\delta.
\end{equation}
\end{definition}

Depending on the type of embedding considered, a priori conditions on $k$ can be obtained to ensure $\Theta$ to be a $(\epsilon, \delta, d)$ oblivious $U \rightarrow \ell_2$ subspace embedding. 
For example, for a Gaussian embedding $\mathbf{\Omega}$ 
(i.e. a matrix whose entries are drawn i.i.d. from a Gaussian distribution with mean $0$ and variance $k^{-1}$), it has been shown in \cite{balabanovRandomizedLinearAlgebra2019} that if $0<\epsilon<0.572$ and 
\begin{equation}
\label{equ:bound gaussian embedding size}
    k \geq 7.87 \epsilon^{-2} (6.9d + \log(\delta^{-1}))
\end{equation} 
then $\mathbf{\Omega}$ is a $(\epsilon, \delta, d)$ oblivious $\ell_2 \rightarrow \ell_2$ subspace embedding. 
Then, by considering $\mathbf{\Theta} = \mathbf{\Omega} \mathbf{Q}$ with $\mathbf{Q}\in\mathbb{R}^{s\times \mathcal{N}}$ such that $\mathbf{Q}\hermtrans \mathbf{Q} = \metricmat$, it follows that $\mathbf{\Theta}$ is a $(\epsilon, \delta, d)$ $U \rightarrow \ell_2$ subspace embedding. 
Note that \eqref{equ:bound gaussian embedding size} is a sufficient condition, which may be very conservative. 
Indeed, numerical experiments from \cite{balabanovRandomizedLinearAlgebra2019, balabanovRandomizedLinearAlgebra2021} showed that random sketching may perform well for much smaller dimension $k$.
It takes $\mathcal{O}(k\mathcal{N})$ operations to sketch a vector with a Gaussian embedding, which can be prohibitive.
A solution is to consider structured embeddings allowing to sketch vectors with only $\mathcal{O}(\mathcal{N}\log(\mathcal{N}))$ operations. 
It is for example the case for the partial subsampled randomized Hadamard transform (P-SRHT) described in \cite{troppImprovedAnalysisSubsampled2011}. 
Sufficient conditions on $k$ are also available but they are very conservative. However, numerical experiments from \cite{balabanovRandomizedLinearAlgebra2019, balabanovRandomizedLinearAlgebra2021} showed performances similar to the Gaussian embedding for a given dimension $k$.
A good practice is to use composed embeddings, for example $\mathbf{\Theta}=\mathbf{\Theta}_2 \mathbf{\Theta}_1$ with $\mathbf{\Theta}_1$ a moderate sized P-SRHT embedding and $\mathbf{\Theta}_2$ a small sized Gaussian embedding.

\begin{remark}
\label{rema:implicite semi inner product matrix}
    In practice, the matrix $\mathbf{Q}$ can be obtained via (sparse) Cholesky factorization of $\metricmat$, but it can also be any rectangular matrix such that $\mathbf{Q}\hermtrans \mathbf{Q} = \metricmat$, as pointed out in \cite[Remark 2.7]{balabanovRandomizedLinearAlgebra2019}.
    Such matrix may be obtained by Cholesky factorizations of small matrices, which are easy to compute. 
    This is especially important for large scale problems.
\end{remark}

\subsection{Randomized selection criterion for parameter-dependent operator equations}
\label{subsec:randomized selection criterion}

Within the framework of \Cref{subsec:parametric pdes}, we propose to replace the function $\mathcal{S}$ by a surrogate quantity  $\mathcal{S}^{\Theta}$, which is a sketched version of $\mathcal{S}$ defined for all $v\in U$ by
\begin{subequations}
\label{equ:definition sketched S2 and sketched residual norm}
\begin{equation}
    \mathcal{S}^{\Theta} (v, \mathcal{P}) 
    := \min_{\xi \in \mathcal{P}} \mathcal{R}^{\Theta}(v, \xi), 
\end{equation}
with 
\begin{equation}
    \mathcal{R}^{\Theta}(v , \xi) 
    := \big\| r(v,\xi) \big\|_{U'}^{\Theta} = \big\|B(\xi)v - f(\xi) \big\|_{U'}^{\Theta}, 
    \quad \forall \xi\in\mathcal{P}.
\end{equation}
\end{subequations}
Let us now discuss how bounds similar to \eqref{equ:residual distance bounds} can be obtained. 
In a general setting where no particular structure is assumed for $B$ or $f$,  $\mathcal{S}^{\Theta}$ can be approximated on a finite (possible very large) set $\mathcal{\widetilde P} \subset \mathcal{P}$. 
\begin{proposition}
\label{prop:sketched distance bounded by distance when finite manifold}
    Assume that $\#\mathcal{\widetilde P}<\infty$. If $\Theta$ is a $(\epsilon, \#\mathcal{\widetilde P}^{-1} \delta, 1)$ oblivious $U\rightarrow \ell_2$ subspace embedding, then for any $v\in U$, with probability at least $1-\delta$ we have
    \begin{equation}
    \label{equ:sketched distance bounded by distance when finite manifold}
        \sqrt{1-\epsilon} ~\mathcal{S}(v, \mathcal{\widetilde P}) 
        \leq \mathcal{S}^{\Theta}(v, \mathcal{\widetilde P}) 
        \leq \sqrt{1+\epsilon} ~\mathcal{S}(v, \mathcal{\widetilde P}).
    \end{equation}
\end{proposition}
\begin{proof}
    For any $v\in U$ and $\xi\in\mathcal{\widetilde P}$, span$\{r(v,\xi)\}$ is a $1$-dimensional space, thus $\Theta$ is a subspace embedding for this space with probability greater than $1 - \# \mathcal{\widetilde P}^{-1} \delta$. 
    Considering a union bound for the probability of failure for each element in $\mathcal{\widetilde P}$ gives the expected result.
\end{proof}
Computational efficiency can be obtained if we assume that $B(\xi)$ is linear and that $B(\xi)$ and $f(\xi)$ admit  affine representations
\begin{equation}
\label{equ:fom affine representation}
    B(\xi) = B^{(0)} + \sum_{q=1}^{m_B} \theta_q^{B}(\xi) B^{(q)},
    \quad 
    f(\xi) = f^{(0)} + \sum_{q=1}^{m_{f}} \theta_q^{f}(\xi) f^{(q)}.
\end{equation}
where $\theta^B_q : \mathcal{P} \rightarrow \mathbb{R}$ and $\theta^{f}_q : \mathcal{P} \rightarrow \mathbb{R}$ are called the affine coefficients, and where $B^{(q)} : U \rightarrow U'$ and $f^{(q)}\in U'$ are parameter-independent and called the affine terms. 
Such representations allow us to do precomputations (offline) independently of $\xi$ using the affine terms, and then to rapidly evaluate (online) the expansion of a parameter dependent quantity. 
This leads to online-efficient computation of the residual norm. 
Those   affine representations may be naturally obtained from the initial problem, such as in the numerical examples in \Cref{sec:numerical examples}, or may be obtained by using for example the empirical interpolation method (EIM) presented in \cite{madayGeneralMultipurposeInterpolation2009}. 

\begin{proposition}
\label{prop:surrogate distance equivalence affine representation}
Assume the affine representations of $B$ and $f$ from \eqref{equ:fom affine representation}. If $\Theta$ is a $(\epsilon, \delta, m_B + m_{f}+1)$ oblivious $U \rightarrow \ell_2$ subspace embedding, then for any $v \in U$, with probability at least $1-\delta$ we have
\begin{equation}
\label{equ:sketched res dist bounded by res dist}
    \sqrt{1-\epsilon} ~\mathcal{S}(v, \mathcal{P}) 
    \leq \mathcal{S}^{\Theta}(v, \mathcal{P}) 
    \leq \sqrt{1+\epsilon} ~\mathcal{S}(v, \mathcal{P}).
\end{equation}
\end{proposition}

\begin{proof}
    For any $v \in U$ and $\xi\in\mathcal{P}$, the residual can be written in a form with separated variables,
    \begin{equation*}
        r(v, \xi) = 
        G(v) \theta(\xi) - g(v) ,
    \end{equation*}
    with
    \begin{equation}
    \label{equ:separated variable formulation}
    \begin{gathered}
        G(v) := \Big( 
        B^{(1)} v \mid \cdots \mid B^{(m_B)} v \mid
        -b^{(1)} \mid \cdots \mid -b^{(m_f)}
        \Big) ,
        \\
        \theta(\xi) = \Big( 
        \theta^B_1(\xi) \mid \cdots \mid \theta^B_{m_B}(\xi) \mid
        \theta^f_1(\xi) \mid \cdots \mid \theta^f_{m_f}(\xi)
        \Big)^T,
        \\
        g(v) := f^{(0)} - B^{(0)} v.
    \end{gathered}
    \end{equation}
    In other words, for a given $v\in U$, $r(v, \xi)$ lies in a low dimensional subspace, spanned by $g(v)$ and the columns of $G(v)$, independently from $\xi$. 
    This subspace is of dimension at most $m_B+m_f+1$. 
    Hence, if $\Theta$ is a $(\epsilon, \delta, m_B + m_f+1)$ oblivious $U \rightarrow \ell_2$ subspace embedding, then for a fixed $v\in U$, with probability at least $1-\delta$,
    \begin{equation*}
        \sqrt{1-\epsilon}~ \mathcal{R}(v, \xi)
        \leq \mathcal{R}^{\Theta}(v, \xi) 
        \leq \sqrt{1+\epsilon}~ \mathcal{R}(v, \xi),
        \hspace{5mm} \forall \xi \in \mathcal{P}.
    \end{equation*}
    Taking the infimum  over $\mathcal{P}$ yields the desired result. 
    \end{proof}

Now that we can ensure a control of $\mathcal{S}^{\Theta}$ by $\mathcal{S}$ with high probability, we can ensure the following error bounds  by using the same reasoning as for obtaining \Cref{prop:general oracle bound}. 
\begin{corollary}
\label{coro:dic-based multi-space sketched error bound 2}
Assume that $\mu(\mathcal{M}, W) < \infty$. If the assumptions of either \Cref{prop:sketched distance bounded by distance when finite manifold} or \Cref{prop:surrogate distance equivalence affine representation} are satisfied, then for any $u \in \mathcal{M}$, with probability at least $1-N \delta,$ we have
\begin{equation}
\label{equ:dic-based multi-space sketched error bound 2}
    \| u - A^{\textup{mult}}_{\mathcal{S}^{\Theta}}(w) \|_U 
    \leq \delta_0 + 2 \kappa^{\Theta} \mu(\mathcal{M}, W) 
    \min_{1\leq k \leq N} \|u - A_{k}(w) \|_U,
\end{equation}
with $\kappa^{\Theta} := \kappa \sqrt{1+\epsilon} / \sqrt{1-\epsilon}$.
\end{corollary}

The key point of this random sketching approach is that we can use a sketch of rather low size while ensuring a desired precision with  high probability. 
Indeed, in view of \eqref{equ:bound gaussian embedding size}, we can satisfy the assumptions of respectively \Cref{prop:sketched distance bounded by distance when finite manifold} or \Cref{prop:surrogate distance equivalence affine representation}
with a sketch of size
\begin{equation*}
    k=\mathcal{O}\Big(\epsilon^{-2} \log(\#\mathcal{\widetilde P}^{-1}\delta^{-1})\Big)
    \hspace{5mm} \text{or} \hspace{5mm}
    k=\mathcal{O}\Big(\epsilon^{-2} \big( (m_B + m_f) + \log(\delta^{-1}) \big)\Big).
\end{equation*}
This will allow us to manipulate only some low dimensional quantities during the online stage, while requiring reasonable offline costs and ensuring robustness to round-off errors. 
This is discussed in more details in \Cref{subsec:offline-online decomposition}.

We end this section by citing \cite{nicholsCoarseReducedModel2022} in which the authors proposed another way of approximating $\mathcal{S}$ using a coarse mesh.
This quantity can be computed efficiently and their numerical experiments show good performances.
However, the authors do not provide a clear way of choosing the size of the coarse mesh in order to satisfy a bound similar to \eqref{equ:distance controled by surrogate}.
On the other hand, we provide explicit bounds on the sketch size $k$ in order to satisfy such a bound with high probability, as stated in \Cref{prop:sketched distance bounded by distance when finite manifold,prop:surrogate distance equivalence affine representation}.
Still, their approach have the advantage of being deterministic and more in line with the use of mesh-based numerical simulations.

\section{Dictionary approach for inverse problem}
\label{sec:dictionary approach for inverse problem}

It has been stated in \cite{balabanovRandomizedLinearAlgebra2021} that the size of the library required for piecewise linear (or affine) reduced modeling may become too large, resulting in prohibitive computational costs. 
The authors then proposed a new dictionary-based approach for forward problems, where low dimensional approximation spaces are obtained from a sparse selection of vectors in a dictionary. 
Most importantly, they made their approach computationally efficient by using randomized linear algebra (RLA) techniques.
We propose to use a similar approach for inverse problems.

First in \Cref{subsec:dictionary-based approximation} we describe the general framework of dictionary-based approximation.
Then in \Cref{subsec:compressive sensing} we summarize a few well known results and algorithms from the compressive sensing literature.
We then propose in \Cref{subsec:dictionary based multispace} to use those and a selection criterion as described in \Cref{subsec:general approach} in order to obtain a non-parametric and non-linear recovery map.
Then in \Cref{subsec:offline-online decomposition} we show that using the randomized criterion we proposed in \Cref{subsec:randomized selection criterion} leads to an efficient offline-online decomposition in the framework of parameter-dependent operator equations (or PDEs) with
affine parametrization.
Finally in \Cref{subsec:limitations} we discuss on the main limitations of the approach we proposed.

\subsection{Dictionary-based approximation}
\label{subsec:dictionary-based approximation}

In this section, following the idea from \cite{balabanovRandomizedLinearAlgebra2021}, we consider a dictionary $\mathcal{D}_K$ containing $K$ normalized vectors in $U$,
\begin{equation}
    \mathcal{D}_K := \{v^{(k)}: 1\leq k\leq K\} 
    \hspace{5mm} \text{with} \hspace{5mm} \|v^{(k)}\|_U = 1, \hspace{5mm} 1\leq k\leq K,
\end{equation}
We then consider the library $\mathcal{L}_n(\mathcal{D}_K)$ defined by 
\begin{equation}
    \mathcal{L}_n(\mathcal{D}_K) := \Big\{ \sum_{k=1}^K x_k v^{(k)} 
    : ~(x_k)_{k=1}^K \in\mathbb{R}^{K},~ \|x\|_0 \leq n \Big\},
\end{equation}
which contains all subspaces spanned by at most $n$ vectors of $\mathcal{D}_K$.
If $\mathcal{D}_K$ is built independently of $\mathcal{M}$, for example a Fourier of wavelet basis, then it corresponds to the classical framework of best $n$-term approximation.
In our context of model order reduction, we can build $\mathcal{D}_K$ so that it is well fitted to $\mathcal{M}$.

The benchmark error for the approximation of $\mathcal{M}$ is then the dictionary-based $n$-width proposed in \cite{balabanovRandomizedLinearAlgebra2021} defined as
\begin{equation}
    \sigma_n(\mathcal{M}, K) := \inf_{\# \mathcal{D}=K}
    \sup_{u\in \mathcal{M}}
    \min_{V \in \mathcal{L}_n(\mathcal{D})}
    \min_{v\in V}\|u- v\|_U.
\end{equation}
The use of the dictionary-based approximation presents two main advantages over methods based on a partitioning of $\mathcal{M}$, such as the piecewise affine approach from \Cref{subsec:piecewise multi-space}. 

Firstly, dictionary-based approximation has a better approximation power than piecewise affine approximation presented in \Cref{subsec:piecewise multi-space}.
Indeed, one can easily see that
\[
\sigma_n(\cdot, nM) \leq d_n(\cdot, M),
\]
while $\mathcal{D}_{nM}$ and $\mathcal{L}_n^M$ both require the storage of $nM$ vectors in $U$.
On the other hand, there exist sets $\mathcal{M}$ such that the piecewise affine approximation requires $\binom{K}{n}$ subspaces to reach the error of the dictionary-based approximation.
For example, one can take $\mathcal{M}$ as the unit ball of $\mathcal{L}_n(\mathcal{D}_K)$, which is compact, and notice that 
\[
0 = d_n \bigg(\mathcal{M}, \binom{K}{n}\bigg) 
\leq 
\sigma_n(\mathcal{M}, K) 
<
d_n\bigg(\mathcal{M}, \binom{K}{n}-1\bigg).
\]

Secondly, \cite{balabanovRandomizedLinearAlgebra2021} shows that $\sigma_n(\cdot, K)$ better preserves the rate of decay with $n$ than $d_n(\cdot, N)_U$ in the case of a superposition of solutions.
It is for example the case for PDEs with multiple transport phenomena. 
More precisely, if $\mathcal{M} \subset \sum_{i=1}^l \mathcal{M}_i$, and if $K=\mathcal{O}(n^{\nu})$ leads to polynomial or exponential decay with $n$ of $\sigma_n(\mathcal{M}_i, K)$, then \cite[Corollary 4.2]{balabanovRandomizedLinearAlgebra2021}  states that $\sigma_n(\mathcal{M}, K)$ has the same type of decay.
Here $l>0$ is some integer and $\nu > 1$ is some small real number.
On the other hand, to obtain similar results with $d_n(\mathcal{M}, N)_U$ defined in \eqref{equ:non linear kolmogorov n width}, one may need a number of subspaces $N \geq n^{l\nu}$.

Finally, we also emphasize the fact that
it does not require any particular parametrization of the initial parametric problem, whereas it can be a central issue for methods based on a partitioning of the parameter space.
However, it is important to note that this problem may also be circumvented by other approaches while remaining in the piecewise-affine approximation setting, as discussed at the end of \Cref{subsec:piecewise multi-space}.

\subsection{Compressive sensing}
\label{subsec:compressive sensing}

Dictionary-based approximation given a few linear measurements on the state can be seen as a compressive sensing problem.
Assuming that $\mathcal{M}$ can be well approximated by $\mathcal{L}_n(\mathcal{D}_K)$, there exists extensively studied algorithms which produce sparse approximations of vectors in $\mathcal{M}$.
For an overview on compressive sensing, we refer to \cite{foucartMathematicalIntroductionCompressive2013}.
Here we focus on the classical constrained $\ell_1$-norm minimization problem, 
\begin{equation}
\label{equ:l1 minimization exact fit}
    \min_{\mathbf{x}\in\mathbb{R}^K} \|\mathbf{x}\|_1
    \quad \text{subject to} \quad
    \mathbf{Cx = w},
\end{equation}
where $\mathbf{C}\in\mathbb{R}^{m\times K}$ is the dictionary-based analogous of the cross gramian matrix from \eqref{equ:pbdw ls formulation algebraic}, defined by $\mathbf{C := W\hermtrans \metricmat V}$ where $\mathbf{V}\in\mathbf{R}^{\mathcal{N}\times K}$ is the matrix whose columns represent the vectors in $\mathcal{D}_K$.
Note that $\mathbf{V}$ is not necessary orthonormal.

However, for denoising and stability purposes, a common approach is to consider instead a relaxed version of \eqref{equ:l1 minimization exact fit},
\begin{equation}
\label{equ:l1 minimization relaxed}
    \min_{\mathbf{x}\in\mathbb{R}^K} \|\mathbf{x}\|_1
    \quad \text{subject to} \quad
    \|\mathbf{Cx - w}\|_2 \leq \lambda,
\end{equation}
for some regularization parameter $\lambda > 0$.
This relaxed problem is closely related to the so called Basis Pursuit Denoising problem, defined by
\begin{equation}
\label{equ:basis pursuit denoising}
    \min_{\mathbf{x} \in \mathbb{R}^K}
    \frac{1}{2} \|\mathbf{C x - w}\|_2^2 + \alpha \|\mathbf{x}\|_1,
\end{equation}
for some regularization parameter $\alpha > 0$.
Indeed \cite[Proposition 3.2]{foucartMathematicalIntroductionCompressive2013} states that a minimizer of \eqref{equ:basis pursuit denoising} is also a minimizer of \eqref{equ:l1 minimization relaxed} for some $\lambda$, and if a minimizer of \eqref{equ:l1 minimization relaxed} is unique then it is also a minimizer of \eqref{equ:basis pursuit denoising} for some $\alpha$.

The problem is that, although we know from \cite[Lemma 14]{tibshiraniLassoProblemUniqueness2013} that there exists at least one $m$-sparse solutions to \eqref{equ:basis pursuit denoising}, if the latter is not unique then there might also exist solutions which are not $m$-sparse.
We thus need to assume that \eqref{equ:basis pursuit denoising} admits a unique solution for all $\alpha >0$, and we denote it by $\mathbf{x}^*_{\alpha}(w)$, which satisfies  $\| \mathbf{x}^*_{\alpha}(w) \|_0 \leq m$ and
\begin{equation}
\label{equ:basis pursuit denoising unique}
    \mathbf{x}^*_{\alpha}(w) 
    := \argmin_{\mathbf{x} \in \mathbb{R}^K}
    \frac{1}{2} \|\mathbf{C x - w}\|_2^2 
    + \alpha \|\mathbf{x}\|_1.
\end{equation}

In the PBDW framework, we shall not focus on $\mathbf{x}^*_{\alpha}(w)$ itself, but rather on its support $\text{supp}(\mathbf{x}^*_{\alpha}(w)) := 
\big\{i\in \{1, \hdots, K\} : \mathbf{x}^*_{\alpha}(w)_i \neq 0 \big\}$,
which allows us to define the reduced space 
$V_{\alpha}(w) \in \mathcal{L}_m(\mathcal{D}_K)$ and the recovery $w \mapsto A_{V_{\alpha}(w)}(w)$ where
\begin{equation}
    V_{\alpha}(w) := 
    \text{span} 
    \{ v^{(i)} : i\in \text{supp} (\mathbf{x}^*_{\alpha}(w)) \}.
\end{equation}
Note that one may also consider using directly the solution $\mathbf{x}^*_{\alpha}(w)$ as background term in \eqref{equ:pbdw ls formulation} and set 
$ v^* = \sum_{i=1}^K \mathbf{x}^*_{\alpha}(w)_i ~ v^{(i)}$.
It is also worth noting that $\alpha \mapsto \mathbf{x}_{\alpha}(w)$ is continuous piecewise-affine.

The problem of this compressive sensing approach is that it relies on a regularization parameter which has to be chosen carefully, using for example a cross validation approach.
In the next subsection, we propose a non-parametric approach by using the general multi-space approach from \Cref{subsec:general approach}, which can be seen as an online selection approach for a parameter $\alpha = \alpha(w)$, or equivalently the corresponding subspace.

\vspace{1em}
\begin{remark}
\label{rmk:uniqueness of lasso}
    There exist sufficient conditions on $\mathbf{C}$ to ensure that \eqref{equ:basis pursuit denoising} admits a unique solution.
    For example \cite[Lemma 4]{tibshiraniLassoProblemUniqueness2013} states that if its entries are drawn i.i.d. from a continuous probability law, then uniqueness is ensured with probability one.
    In our case, this could be ensured by considering a perturbed dictionary $\mathcal{\tilde D}_K$ containing the vectors 
    $\tilde v^{(i)} = v^{(i)} + \sum_{j=1}^m W^{(j)} \mathbf{e}_{i,j} $, where $\mathbf{e}\in\mathbb{R}^{K\times m}$ has entries drawn i.i.d. from a continuous probability law, and $(W^{(j)})_{j=1}^m$ is an orthonormal basis of $W$.
    The cross-gramian $\mathbf{\tilde C}$ associated is then $\mathbf{\tilde C = C + e}$ and satisfies the sufficient condition mentioned above.
    Although one could theoretically make the amplitude of $\mathbf{e}$ arbitrarily small, there is an important and challenging tradeoff to make.
    Indeed, the smaller the $\mathbf{e}$, the closer the approximation power of $\mathcal{\tilde D}_K$ is from the one of $\mathcal{D}_K$, but the more unstable $\alpha \mapsto \mathbf{x}_{\alpha}(w)$ is.
    This discussion needs a deeper investigation, but we will not go further in this work.
\end{remark}

\subsection{Dictionary-based multi-space}
\label{subsec:dictionary based multispace}

One interest in considering problem \eqref{equ:basis pursuit denoising} is stated in \cite[Theorem 15.2]{foucartMathematicalIntroductionCompressive2013}.
Indeed, 
since we assumed that \eqref{equ:basis pursuit denoising} admits a unique solution for all $\alpha>0$, then the LARS-homotopy algorithm \cite{efronLeastAngleRegression2004} efficiently provides solutions $\mathbf{x}^*_{\alpha}(w)$ for all $\alpha > 0$.
Hence, solving \eqref{equ:basis pursuit denoising unique} actually provides a library of spaces $\mathcal{L}(w) \subset \mathcal{L}_m(\mathcal{D}_K)$ defined by 
\begin{equation}
\label{equ:library with lars}
    \mathcal{L}(w) := \{ V_{\alpha}(w) : \alpha > 0 \}.
\end{equation}
Then we propose to use the selection criterion described in \Cref{subsec:general approach} to select a near-optimal subspace among $\mathcal{L}(w)$.
More precisely, for some surrogate distance to the manifold $\mathcal{S}$, we define the dictionary-based multi-space recovery $A^{\text{dic}}_{\mathcal{S}}$ as
\begin{equation}
\label{equ:dictionary based multispace}
    A^{\text{dic}}_{\mathcal{S}} (w) := A_{V_{\mathcal{S}}(w)}
    \quad \text{where} \quad
    V_{\mathcal{S}}(w) := 
    \argmin_{V \in \mathcal{L}(w)}
    \mathcal{S}(A_V(w), \mathcal{M}).
\end{equation}
Although this approach does not test every subspaces in $\mathcal{L}_m(\mathcal{D}_K)$, the results from \Cref{prop:general oracle bound} and \Cref{coro:dic-based multi-space sketched error bound 2} still hold, but with the smaller library $\mathcal{L}(w)$ generated by the LARS-homotopy algorithm. 
This leads us to the following error bound.
\smallskip
\begin{corollary}
\label{coro:LARS error bound}
Assume that $\mu(\mathcal{M}, W) < \infty$. Then for all $u\in\mathcal{M}$ and $w=P_W u$,
\begin{equation}
    \|u - A^{\text{dic}}_{\mathcal{S}}(w)\|_U 
    \leq \delta_0 + 2\kappa\mu(\mathcal{M}, W) 
    \min_{V \in \mathcal{L}(w)} \| u - A_{V}(w) \|_U.
\end{equation}
\end{corollary}

In practice, there are several reasons for which one would actually build a smaller library 
\begin{equation}
\label{equ:def of L hat}
    \mathcal{\hat L}(w) := \{V_{\alpha}(w) : \alpha > \alpha_0(w)\}
\end{equation}
for some $\alpha_0(w)>0$. 
Let us point out two of them.

First, since the LARS-homotopy algorithm computes solutions of \eqref{equ:basis pursuit denoising unique} with decreasing $\alpha$, if the mapping $\alpha \mapsto \mathbf{x}_{\alpha}(w)$ gets too unstable when reaching some $\alpha_0(w)$ then the LARS-homotopy algorithm stops.
This is for example the case when the length of an affine segment of $\alpha \mapsto \mathbf{x}_{\alpha}(w)$ is close to machine precision, or if
$\alpha_0(w)$ is itself close to machine precision.

Second, for online cost reduction one may want to set additional stopping criteria to the LARS-homotopy algorithm. 
One may rather stop the algorithm at some $\alpha_0(w)$ such that
\begin{equation}
\label{equ:def of Nlars}
    \# \hat{\mathcal{L}}(w) \leq N_{\text{LARS}}    
\end{equation} 
with for example $N_{\text{LARS}} = \tau K$ for some small, user-defined, constant $\tau$.
One may also want to stop the algorithm whenever $\|\mathbf{x}^*_{\alpha}(w)\|_0$ gets too large.
A natural criterion would be for example $\|\mathbf{x}^*_{\alpha}(w)\|_0 \leq \frac{m}{2}$ since it is known from \cite[Lemma 3.1]{cohenCompressedSensingBest2008} that one cannot recover exactly all vectors of the form $u=\mathbf{Vx}$ with $\|\mathbf{x}\|_0 > \frac{m}{2}$ from only $m$ linear measurements.

\begin{remark}
\label{rmk:sparsity without unicity}
    We could theoretically get rid of the uniqueness assumption we made in \eqref{equ:basis pursuit denoising unique}.
    Indeed, following \cite[Remark 2]{tibshiraniLassoProblemUniqueness2013} one may modify the LARS-homotopy in order to always output a $m$-sparse solution to problem \eqref{equ:basis pursuit denoising}. 
    However, the computational cost would be much larger, thus efficient implementations of the LARS-homotopy do not include this feature.
\end{remark}

\subsection{Parameter-dependent operator equations: offline-online decomposition}
\label{subsec:offline-online decomposition}

In this subsection, we consider the same framework as in \Cref{subsec:parametric pdes}, in other words problem \eqref{equ:general parametric problem} is actually a parameter-dependent operator equation with inf-sup stable linear operator. 
We first state the randomized version of \Cref{coro:LARS error bound} in the following corollary.

\begin{corollary}
\label{coro:LARS sketched error bound}
Assume that $\mu(\mathcal{M}, W) < \infty$. 
If assumptions of either \Cref{prop:sketched distance bounded by distance when finite manifold} or \Cref{prop:surrogate distance equivalence affine representation} are verified, then for any $u \in \mathcal{M}$ and $w=P_W u$, with probability at least $1-\delta N_{\text{LARS}}$, we have
\begin{equation}
    \|u - A^{\textup{dic}}_{\mathcal{S}^{\Theta}}(w)\|_U 
    \leq \delta_0 + 2\kappa^{\Theta}\mu(\mathcal{M}, W) 
    \min_{V \in \mathcal{\hat L}(w)} \| u - A^{\textup{dic}}_{V}(w) \|_U
\end{equation}
with $\kappa^{\Theta}$ defined in \Cref{coro:dic-based multi-space sketched error bound 2} and $N_{\text{LARS}}$ is from \eqref{equ:def of Nlars}.
\end{corollary}

Let us now discuss on how to efficiently compute $\mathcal{S}^{\Theta}$ when the affine decomposition \eqref{equ:fom affine representation} is assumed. 
In a nutshell, it is done by performing precomputation during an offline stage, independently on the unknown state $u\in\mathcal{M}$ to recover. 
We denote by $\mathbf{U}$ the matrix whose columns span the observation space and the dictionary vectors,
\begin{equation}
    \mathbf{U := \big( ~W\mid V~\big)} \in \mathbb{R}^{\mathcal{N}\times (m+K)}.
\end{equation}
Any dictionary-based recovery can then be expressed as $\mathbf{v}=\mathbf{Ua}$ with $\mathbf{a}\in\mathbb{R}^{m+K}$. 
In the proof of \Cref{prop:surrogate distance equivalence affine representation}, we observed that the computation of $\mathcal{S}$ can be written as a constrained, possibly non linear, least-squares problem. 
It is also the case for $\mathcal{S}^{\Theta}$,  since 
\begin{equation}
\label{equ:sketched distance ls formulation}
    \mathcal{S}^{\Theta}(v, \mathcal{M}) = \min_{\xi\in\mathcal{P}}
    \| G(v) \theta(\xi) - g(v) \|_{U'}^{\Theta}
    = \min_{\xi\in\mathcal{P}}
    \| \mathbf{G}_K^{\Theta}(\mathbf{a}) \theta(\xi) - \mathbf{g}_K^{\Theta}(\mathbf{a}) \|
\end{equation}
with  
\begin{equation}
\label{equ:separated variable formulation reduced sketched}
    \mathbf{G}_K^{\Theta}(\mathbf{a}) := 
    \mathbf{\Theta}\metricmat^{-1} \mathbf{G}_K(\mathbf{a}),
    \quad  
    \mathbf{g}_K^{\Theta}(\mathbf{a}) := \mathbf{\Theta}\metricmat^{-1} \mathbf{g}_K(\mathbf{a}),
\end{equation}
which are the sketched versions of  
\begin{equation}
\label{equ:separated variable formulation reduced}
\begin{gathered}
    \mathbf{G}_K(\mathbf{a}) := \Big( 
    \mathbf{B}^{(1)} \mathbf{Ua} \mid \cdots \mid \mathbf{B}^{(m_B)} \mathbf{Ua} \mid
    -\mathbf{f}^{(1)} \mid \cdots \mid -\mathbf{f}^{(m_f)}
    \Big), \\
    \mathbf{g}_K(\mathbf{a}) := \mathbf{f}^{(0)} - \mathbf{B}^{(0)} \mathbf{Ua},
\end{gathered}
\end{equation}
associated to ${G}$ and ${g}$ defined in \eqref{equ:separated variable formulation}. 
The important point is that we can precompute  $\mathbf{G}^{\Theta}_K(\cdot)$ and $\mathbf{g}_K^{\Theta}(\cdot)$ independently on $\mathbf{a}$. 
More precisely, we compute offline $\mathbf{\Theta} \metricmat^{-1} \mathbf{B}^{(i)}\mathbf{U}$ for $0\leq i\leq m_B$ and $\mathbf{\Theta} \metricmat^{-1} \mathbf{f}^{(j)}$ for $0\leq j\leq m_f$. 
Using some structured embedding, the global offline cost is then typically
\[
    \mathcal{O}\big( ~( m_BK + m_f )~ \mathcal{N}\log(\mathcal{N})~ \big).
\]
Then, for a given observation, we perform the LARS algorithm to obtain a family of  $N_{\text{LARS}}$ state estimates, where we assume that $N_{\text{LARS}}=\mathcal{O}(K)$ in view of \eqref{equ:def of Nlars}, with online cost $\mathcal{O}(m^2 K)$ \cite{efronLeastAngleRegression2004}. 
Now, for each of the vectors $\mathbf{a}$ produced, we can efficiently evaluate $\mathbf{G}^{\Theta}_K(\mathbf{a})$ and $\mathbf{g}_K^{\Theta}(\mathbf{a})$ by performing small matrix-vector products, with matrices of size $k\times (m+K)$. 
More precisely, since the dictionary-based approach produces a sparse background term, the matrices used are only of size $k\times 2m$. 
Therefore, the online preparation of \eqref{equ:sketched distance ls formulation} for every recoveries produced costs $\mathcal{O}(km(m_B+m_f)K)$. 
In summary, if we denote by $C_{ls}$ the cost for  solving a single (nonlinear) least-squares problem \eqref{equ:sketched distance ls formulation}, the total online cost to obtain $A_{\mathcal{S}^{\Theta}}(w)$ is 
\[
    \mathcal{O}\big( 
    \underbrace{\vphantom{C_{ls}} m^2 K}_{\textup{LARS}} + 
    \underbrace{\vphantom{C_{ls}} km(m_B+m_f)K}_{\textup{prepare K l.s.}} + 
    \underbrace{C_{ls}K}_{\textup{solve K l.s.}} 
    \big).
\]
Note that if $\theta(\xi) = \xi\in\mathbb{R}^d$, or more generally if $\theta$ is an affine function of $\xi$, then problem \eqref{equ:sketched distance ls formulation} is a simple linear least-squares, which can be solved at cost $C_{ls} = \mathcal{O}(kd^2)$. 
Otherwise, the cost depends on the global optimization procedure used. 
We recall that, in view of \eqref{equ:bound gaussian embedding size}, we have typically 
\[
    k = \mathcal{O}\Big(\epsilon^{-2} \big( (m_B + m_f) + \log(K\delta^{-1})\big)\Big),
\]
with $\epsilon$ the precision of the sketch and $\delta$ the probability of failure from \Cref{def:oublivous embedding def}. 
We recall that both are user-defined.

The advantage of considering $\mathcal{S}^{\Theta}$ instead of $\mathcal{S}$ is that it yields a small least-squares system, which is efficiently prepared online, whereas doing a similar procedure with $\mathcal{S}$ leads to a large least-squares system of size $\mathcal{N}\times (m_B+m_f)$. 
A common approach to compute the residual norm efficiently is to expand offline $\|r(\mathbf{Ua}, \cdot)\|^2_{U'}$ as 
\begin{equation*}
    \|r(\mathbf{Ua};\xi)\|_{U'}^2 = 
    \mathbf{a}\hermtrans \mathbf{M}_1 (\xi) \mathbf{a}
    - 2 \mathbf{a}\hermtrans \mathbf{M}_2(\xi)
    + \mathbf{M}_3(\xi),
    \quad \forall \xi \in \mathcal{P},
\end{equation*}
where $\mathbf{M}_1: \mathcal{P} \rightarrow \mathbb{R}^{(K+m)\times (K+m)}$, $\mathbf{M}_2 : \mathcal{P} \rightarrow \mathbb{R}^{(K+m)}$ and $\mathbf{M}_3 : \mathcal{P} \rightarrow \mathbb{R}$ all admit affine representations, with respectively $m_B^2$, $m_B m_f$ and $m_f^2$ affine terms. 
Although this expression is theoretically exact, in practice the high number of affine terms makes this approach more sensitive to round-off errors, as pointed out in \cite{casenaveAccurateOnlineefficientEvaluation2014, balabanovRandomizedLinearAlgebra2019, buhrNumericallyStablePosteriori2014}, and the associated offline cost scales as 
\[
    \mathcal{O}\big( ~( m_BK + m_f )^2~ \mathcal{N}~ \big)
\]
which is   prohibitive since we aim for large $K$. Another approach proposed \cite{buhrNumericallyStablePosteriori2014} is to compute an orthonormal basis of the space in which the residual lies. 
This approach is more stable than the previous one, but comes with a similar prohibitive offline cost. 
Note that another approach is proposed in \cite{casenaveAccurateOnlineefficientEvaluation2014,giraldiWeaklyIntrusiveLowRank2019}, where different empirical interpolation methods (EIM) are proposed to approximate the residual norm. 

\subsection{Limitations}
\label{subsec:limitations}

Although there are several advantages of the dictionary-based approach as stated in \Cref{subsec:dictionary-based approximation}, we end this subsection by emphasizing the main limitations of our approach.

Firstly, the selection criterion $\mathcal{S}$ we rely on is based on a model, typically a discretization of a continuous PDE.
If the model error is large, then most probably using this selection criterion is not relevant.
In such cases, which actually correspond to many practical cases, it may be more relevant to use \eqref{equ:basis pursuit denoising unique} with a fixed parameter $\alpha$, fitted using for example a cross-validation procedure, as considered in \cite{bruntonDiscoveringGoverningEquations2016,kaiserSparseIdentificationNonlinear2018,callahamRobustFlowReconstruction2019}.
Note that those works used directly $\mathbf{x}_{\alpha}$ for the coefficients of the state estimate, whereas here we considered the space spanned by its support.

Secondly, we can not control the constant $\mu(V,W)$ for every space $V\in\mathcal{L}_m(\mathcal{D}_K)$.
Therefore, we can not bound the best recovery error contrary to the piecewise multi-space setting described in \Cref{subsec:piecewise multi-space}, with a smaller number of spaces involved.
Similarly, there is no known way to compute $\mu(\mathcal{M}, W)$ for general $\mathcal{M}$.

Thirdly, if for some $\alpha$ the problem \eqref{equ:basis pursuit denoising} does not have a unique solution, then there might exist solutions which are not $m$-sparse.
Although this issue may be circumvented, this comes with challenging tradeoffs, as discussed in \Cref{rmk:uniqueness of lasso,rmk:sparsity without unicity}.

\section{Numerical examples}
\label{sec:numerical examples}

To illustrate the performances of our dictionary-based multi-space approach, we provide two numerical examples where problem \eqref{equ:general parametric problem} is a parametric PDE defined on some open, regular, bounded domain $\Omega \subset \mathbb{R}^2$. 
The solution space is the Hilbert space $\mathcal{H}_{\Gamma_D}^1(\Omega) = \{u \in \mathcal{H}^1(\Omega) : u =0 \; \text{on} \; \Gamma_D\}$, where $\Gamma_D$ is the  boundary's part where an homogeneous Dirichlet condition is applied. 
The problem is discretized using  Lagrange $\mathbb{P}_1$ finite elements with a regular triangular mesh, leading to a finite dimensional space $U$ that we equip with the natural norm  $\|\cdot\|_U = \| \nabla \cdot \|_{\mathcal{L}^2(\Omega)}$ in $\mathcal{H}_{\Gamma_D}^1(\Omega)$.

We consider local radial integral sensors, approximated in the finite element space, such that
\begin{equation}
\label{equ:thermal block sensors}
    \ell_i(u) \simeq \int_{\Omega} u(x) 
    \exp \Big( -\frac{\|x-x^{(i)}_{\text{loc}} \|^2}{\sigma^2}\Big) dx,
    \quad  1\leq i\leq m,
\end{equation}
where $x^{(i)}_{\text{loc}} \in\Omega$ is the location of the $i$-th sensor and $\sigma$ represent the filter width of the sensor. 

To approximate the residual-based distance $\mathcal{S}$ with $\mathcal{S}^{\Theta}$, we consider random embeddings of the form $\mathbf{\Theta}=\mathbf{\Omega} \mathbf{Q}$ with $\mathbf{\Omega}\in\mathbb{R}^{k\times \mathcal{N}}$ a random embedding and $\mathbf{Q}\in\mathbb{R}^{\mathcal{N}\times \mathcal{N}}$ a matrix such as $\mathbf{Q\hermtrans Q}=\metricmat$, obtained by Cholesky factorization.
Note that in both examples, the quantity $\mathcal{S}^{\Theta}(\cdot, \mathcal{P})$ can be exactly computed by solving a constrained linear least-squares system.

We will compare the performances of three types of recovery maps, denoted by $A^{(1)}$, $A^{(2)}$ and $A^{(3)}$, on a test set of $N_{\text{test}}=500$ snapshots for various dictionary sizes $K$. 
More precisely, we compare the recovery relative error, $ \|u-A^{(j)}(w)\|_U / \|u\|_U$ for $j=1,2,3$. 
The first recovery map is the best multi-space recovery map from \Cref{subsec:nested multi-space} with a library containing $m$ nested subspaces obtained by performing a POD on the $K$ first vectors of $\mathcal{M}_{\text{prior}}$, which is a discrete set of snapshots.
It is defined such that 
\begin{equation}
\label{equ:best adaptive pod}
    \| A^{(1)}(w) - u\|_U:=
    \min_{1\leq k \leq m} \|A_{V^{\text{POD}}_k}(w) - u\|_U.
\end{equation}
The second one is the dictionary-based recovery map $A^{(2)} := A^{\text{dic}}_{\mathcal{S}^{\Theta}}$ as defined in \eqref{equ:dictionary based multispace}, with the randomized selection criterion defined in \eqref{equ:definition sketched S2 and sketched residual norm}.
The third one is the best dictionary-based recovery map with optimal space selected among the adaptive library $\mathcal{\hat L}(w)$ defined in \eqref{equ:def of L hat}. 
In other words,
\begin{equation}
\label{equ:best lars recovery}
    \| A^{(3)}(w) - u\|_U:=
    \min_{V \in \mathcal{\hat L}(w)} \|A_{V}(w) - u\|_U.
\end{equation}
For both dictionary-based recoveries, the dictionary $\mathcal{D}_K$ used is composed of the $K$ first (normalized) snapshots from $\mathcal{M}_{\text{prior}}$.

In \Cref{subsec:thermal block} we consider a diffusion problem, built with the pymor library \cite{milkPyMORGenericAlgorithms2016}, which provides many standard tools for model order reduction. Then in \Cref{subsec:advection-diffusion} we consider an advection-diffusion problem, built using the DOLFINx package \cite{loggDOLFINAutomatedFinite2010}. 
Finally in \Cref{subsec:observations} we summarize our main observations. 

The implementation has been made via an open-source python development available at \url{https://github.com/alexandre-pasco/rla4mor/tree/main/inverse_problems}. 
The implementation of the sparse Cholesky factorization is from the scikit-sparse package, which includes a wrapper for the CHOLMOD package \cite{chenAlgorithm887CHOLMOD2018}.  
The implementation of the LARS algorithm is from the SPAMS toolbox library \cite{mairalOnlineDictionaryLearning2009}.

\subsection{Thermal block}
\label{subsec:thermal block}

In this subsection we consider a classical thermal block problem with Dirichlet boundary conditions. 
It models a heat diffusion in a 2D domain $\Omega = (0,1)^2$ composed of $9$ blocks $\Omega_1, \hdots, \Omega_9$ of equal sizes and various thermal conductivities $\xi_1, \cdots, \xi_9$. 
The heat comes from a constant volumic source term equal to $1$ and homogeneous Dirichlet boundary conditions are imposed on $\Gamma_D = \partial \Omega$. The problem we consider can be written as
\begin{align}
\label{equ:thermal block}
\left\{\begin{aligned}
    -\nabla \cdot \big(\kappa\nabla u\big) =~& 1 &\text{in } &\Omega, \\
    u =~& 0 &\text{on } &\Gamma_D, \\
    \kappa =~& \xi_i &\text{in } &\Omega_i, ~1\leq i\leq 9,
\end{aligned}\right.
\end{align}
where the diffusion coefficients $\xi_i$ are drawn independently in $[\frac{1}{10}, 1]$ with a log-uniform distribution, leading to a parameter set $\mathcal{P} = [\frac{1}{10}, 1]^{9}$. 
The maximal finite element mesh element diameter is $2^{-6}$, leading to $\mathcal{N}=8~321$ degrees of freedom. 
This leads to the simple affine representations 
\begin{equation}
\label{equ:fom thermal}
    \mathbf{B}(\xi) = \mathbf{B}^{(0)} + \sum_{i=1}^9 \xi_i \mathbf{B}^{(i)},
    \hspace{5mm}
    \mathbf{f}(\xi) = \mathbf{f}.
\end{equation}
We consider $64$ sensors uniformly placed  at positions $ (\frac{i}{9}, \frac{j}{9})_{1\leq i,j \leq 8}$, with a sensor width $\sigma=2^{-6}$. 
The problem can be visualized in \Cref{fig:thermal block problem}.

\begin{figure}[!ht]
    \centering
    \includegraphics[page=3, width=1\textwidth]{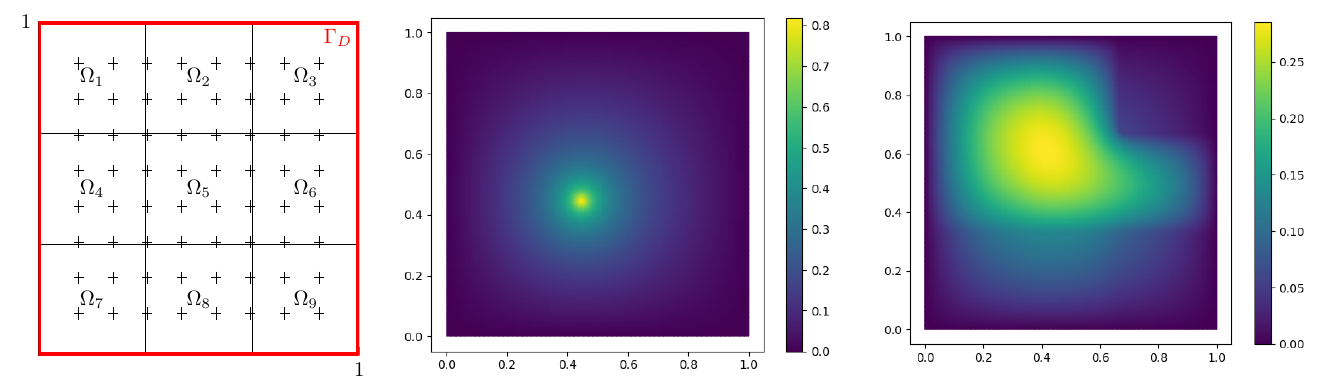}
    \caption[Thermal block problem]{The thermal block problem.
    On the left, the geometry of the problem, with sensors locations (crosses). 
    On the middle, the (normalized) Riesz representer of a sensor. 
    On the right, an example of snapshot.}
    \label{fig:thermal block problem}
\end{figure}

We consider $3$ different values for $m$: $64$, $36$ and $9$. 
For $m=64$, all the sensors are used, for $m=36$ only sensors at positions $(\frac{i}{9}, \frac{j}{9})_{i,j \in \{1,2,4,5,7,8\}}$ are used, and for $m=9$ only sensors at positions $(\frac{i}{9}, \frac{j}{9})_{i,j \in \{1,4,7\}}$ are used. 
We computed a set of snapshots $\mathcal{M}_{\text{prior}}$ of size $5000$ to form our prior knowledge on the manifold $\mathcal{M}$. 

First in \Cref{fig:thermal block pod eps mu} we show the
evolution of the constants $\varepsilon_n$, $\mu(V_n, W)$ and the product $\varepsilon_n \mu(V_n,W)$ involved in the error bound \eqref{equ:pbdw error bound}, where $V_n$ is the space spanned by the $n$ first POD modes computed on $\mathcal{M}_{\text{prior}}$, and with $m=64$ measurements.
The constant $\varepsilon_n$ is approximated on a test set of $500$ snapshots by $\varepsilon_n \simeq \max_{u \in \mathcal{M}_{\text{test}}} \|u - P_{V_n} u\|_U$, and the constant $\mu(V_n, W)$ is computed exactly by SVD.

This figure shows the major issue of PBDW, which is the trade off that has to be made between the approximation of $\mathcal{M}$ by a reduced space $V_n$, i.e. low $\varepsilon_n$, and the angle between $V_n$ and the observation space $W$, i.e. low $\mu(V_n, W)$.
This issue is even more problematic when the $V_n$ and $W$ are built independently, as it is the case here.

\begin{figure}[!ht]
    \centering
    \includegraphics[page=7, height=0.4\textwidth]{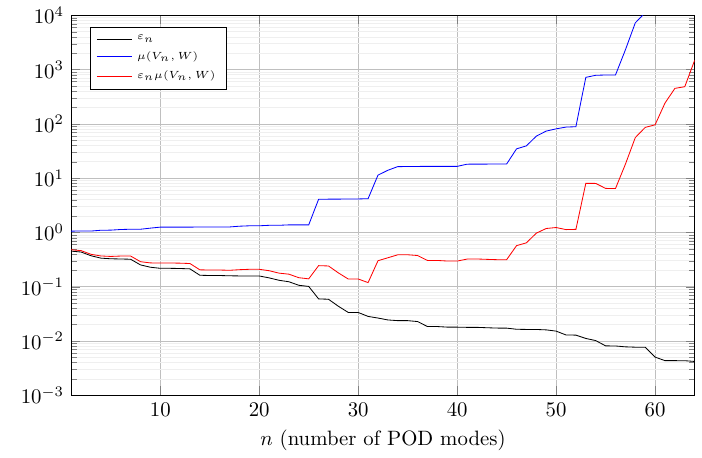}
    \caption[Thermal block pod eps mu]{
    Evolution of the constants in the error bound \eqref{equ:pbdw error bound}, on $500$ test snapshots, with increasing number of POD modes and $m=64$ measurements.
    }
    \label{fig:thermal block pod eps mu}
\end{figure}

Then in \Cref{fig:thermal block results}, we test the performances of our dictionary-based approach with dictionaries of size $K \in \{ 100,200,500,1000,2000,5000 \}$. 
We choose $\mathbf{\Omega}$ as a Gaussian embedding with $k=100$ rows.
Note that since $\mathcal{N}$ is rather small, we could compute $\mathcal{S}$ online by solving the least-squares system of size $\mathcal{N} \times d$ with rather low computational cost.

We first observe that the dictionary-based approaches $A^{(2)}$ and $A^{(3)}$ both outperform $A^{(1)}$ in the mean error sense, for all values of $m$ and $K$ considered.
For the worst case error, so does $A^{(3)}$.
However for $m=9$ it takes a large $K$ for $A^{(2)}$ to outperform $A^{(1)}$.
This shows that the library $\mathcal{\hat L}(w)$, defined in \eqref{equ:library with lars}, generated by the compressive sensing approach has a good potential, but the selection criterion with $\mathcal{S}^{\Theta}$ is, as expected, not optimal.

Finally, we observe a larger performance gap between $A^{(2)}$ and $A^{(3)}$ for $m=9$ than for $m=36,64$.
This may be interpreted as a degradation of the constant $\mu(\mathcal{M}, W)$ in \Cref{coro:LARS error bound}, although it is only speculative.

From those observations, we conclude on better global performances of $A^{(2)}$ compared to $A^{(1)}$.

\begin{figure}[!ht]
    \centering
    \includegraphics[page=5, width=1\textwidth]{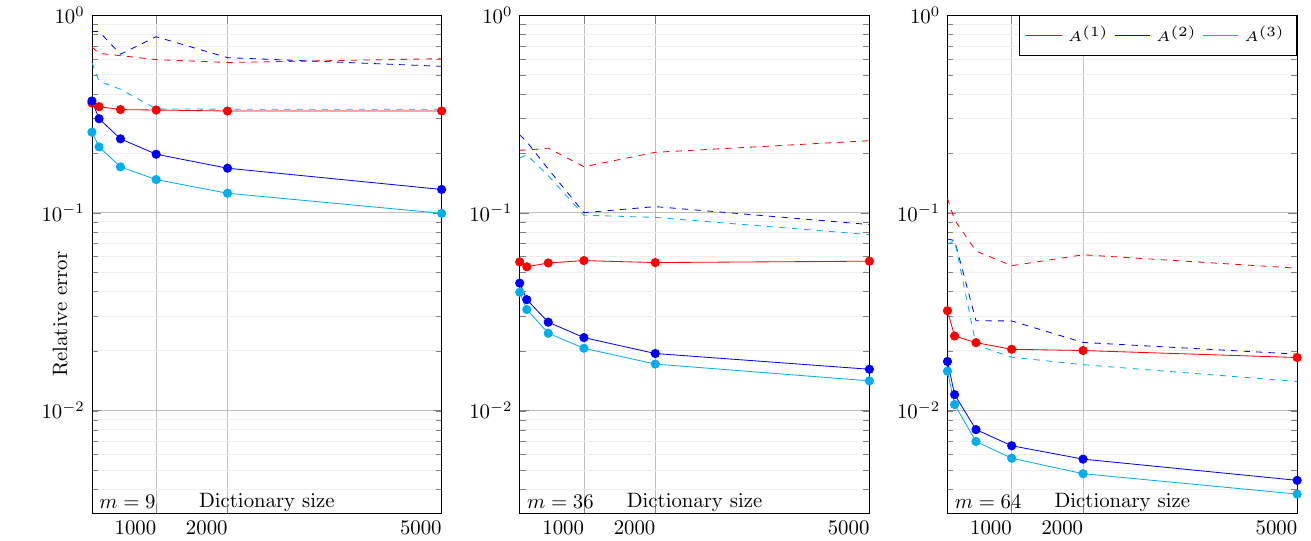}
    \caption[Thermal block results]{Evolution of the recovery errors in $U$-norm, on $500$ test snapshots, with growing dictionary sizes and different values of $m$. 
    We compare the PBDW recovery based on the best adaptive POD truncation  $A^{(1)}$, in red, defined in \eqref{equ:best adaptive pod}, the dictionary-based recovery with randomized selection criterion $A^{(2)} = A^{\text{dic}}_{\mathcal{S}^{\Theta}}$, in blue, as well as the best one produced by the LARS algorithm $A^{(3)}$, in cyan, defined in \eqref{equ:best lars recovery} . 
    The full line is the mean relative error and the dotted line is the maximal relative error.
    }
    \label{fig:thermal block results}
\end{figure}

\subsection{Advection-diffusion}
\label{subsec:advection-diffusion}

In this subsection we consider an advection diffusion problem with multiple transport phenomena, inspired from \cite{balabanovRandomizedLinearAlgebra2021}. 
It models a heat diffusion and advection in a 2D domain $\Omega = \mathcal{B}(0,1.5) \setminus \bigcup_{i=1}^5 \Omega_i$, with $\Omega_i=\mathcal{B}(x^{(i)}, 0.1)$ a ball of radius $0.1$ centered at $x^{(i)}=(\cos(2i\pi/5), \sin(2i\pi/5))$, for $1\leq i\leq 5$. 
A constant source term is considered on the circular subdomain $\Omega_S = \mathcal{B}(0,0.1)$. 
The (parametric) advection field $\mathcal{V}$ is a sum of $5$ potential fields around the domains $\Omega_i$. 
The unknown temperature field $u$ satisfies the equations
\begin{align}
\label{equ:advection diffusion problem}
\left\{\begin{aligned}
    -\kappa \Delta u  + \mathcal{V}(\xi) \cdot \nabla u =~& \frac{100}{\pi} \mathbbm{1}_{\Omega_S} &\text{in } &\Omega, \\
    u =~& 0 &\text{on } &\Gamma_D, \\
    n \cdot \nabla u =~& 0 &\text{on } &\Gamma_N,
\end{aligned}\right.
\end{align}
with a diffusion coefficient $\kappa = 0.01$ and the advection field
\[
    \mathcal{V}(\xi) = 
    \sum_{i=1}^5 \frac{1}{\|x-x^{(i)}\|} \Big(
    \xi_i e_r(x^{(i)}) 
    + \xi_{i+5} e_{\theta}(x^{(i)})
    \Big),
\]
where $e_r(x^{(i)})$ and $e_{\theta}(x^{(i)})$ are the basis vectors of the polar coordinate system centered at $x^{(i)}$. 
Dirichlet boundary conditions are imposed on $\Gamma_D = \partial \mathcal{B}(0,1.5)$ and homogeneous Neumann conditions are imposed on $\Gamma_N = \Omega \setminus \Gamma_D$. 
The parameter $\xi$ is chosen uniformly in $\mathcal{P}=[-1,-\frac{1}{2}]^5 \times [-2, -1]^5$. 
The problem is discretized using Lagrange $\mathbb{P}_1$ finite elements on a triangular mesh refined around the pores, leading to  $\mathcal{N} = 152~297$ degrees of freedom. 
The operator and right-hand admit the  affine representations 
\begin{equation}
\label{equ:fom advection diffusion}
    \mathbf{B}(\xi) = \mathbf{B}^{(0)} + \sum_{i=1}^{10} \xi_i \mathbf{B}^{(i)},
    \hspace{5mm}
    \mathbf{f}(\xi) = \mathbf{f}.
\end{equation}
We consider $101$ sensors.
A first sensor is placed at point $(0,0)$, and for $1\leq j\leq 4$,  $10\times j$ sensors are placed uniformly on a circle of radius $0.2\times j$ centered at $(0,0)$.
The sensor width is chosen as $\sigma=2.10^{-2}$. 
The problem can be visualized in \Cref{fig:advection diffusion problem}. 

\begin{figure}[!ht]
    \centering
    \includegraphics[page=4, width=1\textwidth]{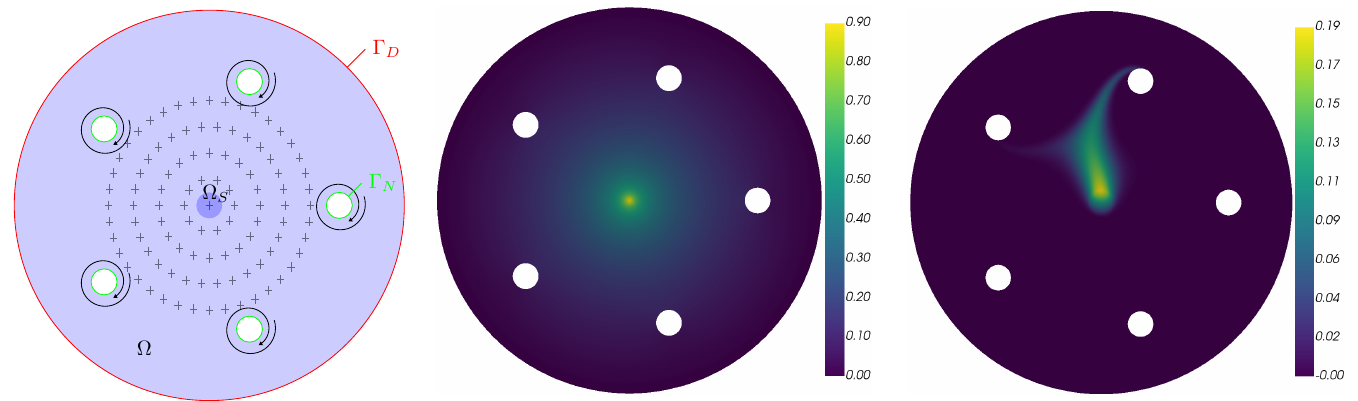}
    \caption[Advection diffusion problem]{The advection diffusion problem. 
    On the left, the geometry of the problem, with sensors locations (crosses) and advection fields (circular arrows).
    On the middle, the (normalized) Riesz representer of the central sensor. 
    On the right, an example of snapshot.}
    \label{fig:advection diffusion problem}
\end{figure}

We consider 3 different values for $m$: $101$, $61$ and $31$.
For $m=101$ every sensors are used, for $m=61$ we dropped the sensors of the outer circle, and for $m=31$ we dropped the sensors of the two outer circles. 
We computed a set $\mathcal{M}_{\text{prior}}$ of $2048$ snapshots to form our prior knowledge on the manifold $\mathcal{M}$. 

First in \Cref{fig:advection diffusion pod eps mu} we show the
evolution of the constants $\varepsilon_n$, $\mu(V_n, W)$ and the product $\varepsilon_n \mu(V_n,W)$ involved in the error bound \eqref{equ:pbdw error bound}, where $V_n$ is the space spanned by the $n$ first POD modes computed on $\mathcal{M}_{\text{prior}}$, and with $m=101$ measurements.
The constant $\varepsilon_n$ is approximated on a test set of $500$ snapshots by $\varepsilon_n \simeq \max_{u \in \mathcal{M}_{\text{test}}} \|u - P_{V_n} u\|_U$, and the constant $\mu(V_n, W)$ is computed exactly by SVD.

This figure emphasizes even more the issue exhibited in \Cref{fig:thermal block pod eps mu}.
Here, $\varepsilon_n$ decays much more slowly than $\mu(V_n, W)$ grows.
Thus the trade-off for the error bound \eqref{equ:pbdw error bound} leads to a very low $n$.

\begin{figure}[!ht]
    \centering
    \includegraphics[page=8, height=0.4\textwidth]{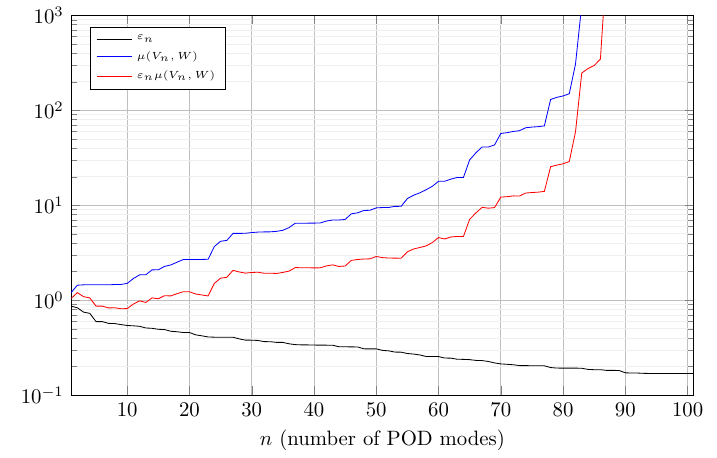}
    \caption[Thermal block pod eps mu]{
    Evolution of the constants in the error bound \eqref{equ:pbdw error bound}, on $500$ test snapshots, with growing POD truncation and $m=64$ measurements.
    }
    \label{fig:advection diffusion pod eps mu}
\end{figure}

Then in \Cref{fig:advection diffusion results} we test the performances of our dictionary-based approach with dictionaries of size $K \in \{ 256,512,1024,2048 \}$. 
We choose $\mathbf{\Omega} = \mathbf{\Omega}_2 \mathbf{\Omega}_1$ with $\mathbf{\Omega}_1$ a P-SRHT embedding with $k'=16~263$ rows and $\mathbf{\Omega}_2$ a Gaussian embedding with $k=850$ rows.
It is important to note that with the rather large $\mathcal{N}$ and $K$ involved, the randomized approach from \Cref{sec:randomized multi-space problem} was crucial to efficiently evaluate online $\mathcal{S}^{\Theta}$ while requiring a reasonable offline cost, which may not possible with $\mathcal{S}$. 

We first observe that the dictionary-based approaches $A^{(2)}$ and $A^{(3)}$ both outperform $A^{(1)}$ in the mean error sense, for all values of $m$ and $K$ considered.
For the worst case error, it takes a dictionary of size $K=2048$
for $A^{(3)}$ to outperform $A^{(1)}$ for all values of $m$, otherwise the performances are similar.
A similar behavior is observed for $A^{(2)}$, except that for small values of $K$, $A^{(1)}$ performs better.
Again, this shows that $\mathcal{\hat L}(w)$ has a good potential, at least when $K$ is large enough, but the selection criterion with $\mathcal{S}^{\Theta}$ is, as expected, not optimal.

Finally, we observe that the performance gap between $A^{(2)}$ and $A^{(3)}$ decreases when $m$ grows.
Again, this may be interpreted as a degradation of the constant $\mu(\mathcal{M}, W)$ in \Cref{coro:LARS error bound}, although it is still only speculative.

From those observations, we conclude on better global performances of $A^{(2)}$ compared to $A^{(1)}$, especially for large enough dictionaries.

\begin{figure}[!ht]
    \centering
    \includegraphics[page=6, width=1\textwidth]{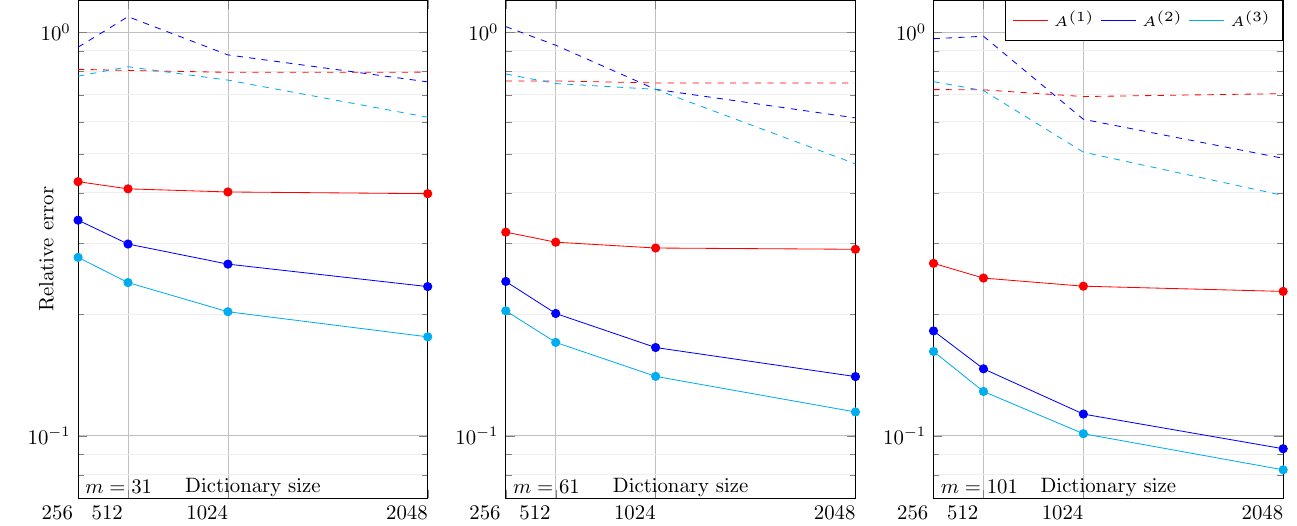}
    \caption[Advection diffusion results]{Evolution of the recovery errors in $U$-norm, on $500$ test snapshots, with growing dictionary sizes, and different values of $m$.
    We compare the PBDW recovery based on the best adaptive POD truncation  $A^{(1)}$, in red, defined in \eqref{equ:best adaptive pod}, the dictionary-based recovery with randomized selection criterion $A^{(2)} = A^{\text{dic}}_{\mathcal{S}^{\Theta}}$, in blue, as well as the best one produced by the LARS algorithm $A^{(3)}$, in cyan, defined in \eqref{equ:best lars recovery}. 
    The full line is the mean relative error and the dotted line is the maximal relative error.
    }
    \label{fig:advection diffusion results}
\end{figure}

\subsection{Summary of observations}
\label{subsec:observations}

From the numerical results of the previous subsections, we draw two main observations. 

Firstly, our dictionary-based approach with randomized selection criterion outperformed the best possible adaptive POD-based PBDW recovery in the mean error sense, especially with rather large dictionaries.
While it was not always the case for the worst case error, especially with small dictionaries and few observations, we observed that it would have been the case with an optimal selection within $\mathcal{L}(w)$. 
As a result, we conclude that our approach is truly interesting when the dictionary considered is rather large. 

Secondly, we observed that the optimality of the randomized selection criterion $\mathcal{S}^{\Theta}$ deteriorates when $m$ is too small.
This may be explained by an increase of $\mu(\mathcal{M}, W)$ in \Cref{coro:LARS error bound}, although we are not able to validate this hypothesis.

\section{Conclusions}
\label{sec:conclusions}

In this work, we proposed a dictionary-based model reduction approach for inverse problems, similar to what already exists for direct problems, with a near-optimal subspace selection based on some surrogate distance to the solution manifold, among the solutions of a set of $\ell_1$-regularized problems. 
We focused on the framework of parameter-dependent operator equations with affine parametrization, for which we provided an efficient procedure based on randomized linear algebra, ensuring stable computation while preserving theoretical guarantees.

Future work shall study the performances of our approach in a noisy observations framework. 
One may also consider a bi-dictionary approach, in which both the observation space and the background space would be selected adaptively with a dictionary-based approach.

\section*{Conflict of Interest}
The authors declared that they have no conflict of interest.

\bibliographystyle{plain}  
\bibliography{main}

\begin{thebibliography}{10}

\bibitem{amsallemPEBLROMProjectionerrorBased2016}
David Amsallem and Bernard Haasdonk.
\newblock {{PEBL-ROM}}: {{Projection-error}} based local reduced-order models.
\newblock {\em Adv. Model. and Simul. in Eng. Sci.}, 3(1):6, December 2016.

\bibitem{amsallemNonlinearModelOrder2012}
David Amsallem, Matthew~J. Zahr, and Charbel Farhat.
\newblock Nonlinear model order reduction based on local reduced-order bases.
\newblock {\em Numerical Meth Engineering}, 92(10):891--916, December 2012.

\bibitem{amsallemFastLocalReduced2015}
David Amsallem, Matthew~J. Zahr, and Kyle Washabaugh.
\newblock Fast local reduced basis updates for the efficient reduction of
  nonlinear systems with hyper-reduction.
\newblock {\em Adv Comput Math}, 41(5):1187--1230, October 2015.

\bibitem{balabanovRandomizedLinearAlgebra2019}
Oleg Balabanov and Anthony Nouy.
\newblock Randomized linear algebra for model reduction. {{Part I}}:
  {{Galerkin}} methods and error estimation.
\newblock {\em Adv Comput Math}, 45(5-6):2969--3019, December 2019.

\bibitem{balabanovPreconditionersModelOrder2021}
Oleg Balabanov and Anthony Nouy.
\newblock Preconditioners for model order reduction by interpolation and random
  sketching of operators.
\newblock {\em arXiv:2104.12177 [cs, math]}, April 2021.

\bibitem{balabanovRandomizedLinearAlgebra2021}
Oleg Balabanov and Anthony Nouy.
\newblock Randomized linear algebra for model reduction\textemdash part {{II}}:
  Minimal residual methods and dictionary-based approximation.
\newblock {\em Adv Comput Math}, 47(2):26, April 2021.

\bibitem{binevDataAssimilationReduced2017}
Peter Binev, Albert Cohen, Wolfgang Dahmen, Ronald DeVore, Guergana Petrova,
  and Przemyslaw Wojtaszczyk.
\newblock Data {{Assimilation}} in {{Reduced Modeling}}.
\newblock {\em SIAM/ASA J. Uncertainty Quantification}, 5(1):1--29, January
  2017.

\bibitem{binevGreedyAlgorithmsOptimal2018}
Peter Binev, Albert Cohen, Olga Mula, and James Nichols.
\newblock Greedy {{Algorithms}} for {{Optimal Measurements Selection}} in
  {{State Estimation Using Reduced Models}}.
\newblock {\em SIAM/ASA J. Uncertainty Quantification}, 6(3):1101--1126,
  January 2018.

\bibitem{bruntonDiscoveringGoverningEquations2016}
Steven~L. Brunton, Joshua~L. Proctor, and J.~Nathan Kutz.
\newblock Discovering governing equations from data by sparse identification of
  nonlinear dynamical systems.
\newblock {\em Proc. Natl. Acad. Sci. U.S.A.}, 113(15):3932--3937, April 2016.

\bibitem{buhrNumericallyStablePosteriori2014}
Andreas Buhr, Christian Engwer, Mario Ohlberger, and Stephan Rave.
\newblock A numerically stable a posteriori error estimator for reduced basis
  approximations of elliptic equations.
\newblock {\em preprint}, 2014.

\bibitem{callahamRobustFlowReconstruction2019}
Jared~L. Callaham, Kazuki Maeda, and Steven~L. Brunton.
\newblock Robust flow reconstruction from limited measurements via sparse
  representation.
\newblock {\em Phys. Rev. Fluids}, 4(10):103907, October 2019.

\bibitem{casenaveAccurateOnlineefficientEvaluation2014}
Fabien Casenave, Alexandre Ern, and Tony Leli{\`e}vre.
\newblock Accurate and online-efficient evaluation of the {\emph{a posteriori}}
  error bound in the reduced basis method.
\newblock {\em ESAIM: M2AN}, 48(1):207--229, January 2014.

\bibitem{chenAlgorithm887CHOLMOD2018}
Yanqing Chen, Timothy~A Davis, William~W Hager, and Sivasankaran Rajamanickam.
\newblock Algorithm 887: {{CHOLMOD}}, {{Supernodal Sparse Cholesky
  Factorization}} and {{Update}}/{{Downdate}}.
\newblock {\em ACM Transactions on Mathematical Software}, 35(3):14, October
  2018.

\bibitem{cohenCompressedSensingBest2008}
Albert Cohen, Wolfgang Dahmen, and Ronald DeVore.
\newblock Compressed sensing and best k-term approximation.
\newblock {\em J. Amer. Math. Soc.}, 22(1):211--231, July 2008.

\bibitem{cohenOptimalReducedModel2020a}
Albert Cohen, Wolfgang Dahmen, Ronald DeVore, Jalal Fadili, Olga Mula, and
  James Nichols.
\newblock Optimal {{Reduced Model Algorithms}} for {{Data-Based State
  Estimation}}.
\newblock {\em SIAM J. Numer. Anal.}, 58(6):3355--3381, January 2020.

\bibitem{cohenNonlinearReducedModels2022}
Albert Cohen, Wolfgang Dahmen, Olga Mula, and James Nichols.
\newblock Nonlinear {{Reduced Models}} for {{State}} and {{Parameter
  Estimation}}.
\newblock {\em SIAM/ASA J. Uncertainty Quantification}, 10(1):227--267, March
  2022.

\bibitem{cohenNonlinearApproximationSpaces2022a}
Albert Cohen, Matthieu Dolbeault, Olga Mula, and Agustin Somacal.
\newblock Nonlinear approximation spaces for inverse problems, October 2022.

\bibitem{drohmannAdaptiveReducedBasis2011}
Martin Drohmann, Bernard Haasdonk, and Mario Ohlberger.
\newblock Adaptive {{Reduced Basis Methods}} for {{Nonlinear
  Convection}}{\textendash}{{Diffusion Equations}}.
\newblock In Jaroslav Fo{\v r}t, Ji{\v r}{\'i} F{\"u}rst, Jan Halama,
  Rapha{\`e}le Herbin, and Florence Hubert, editors, {\em Finite {{Volumes}}
  for {{Complex Applications VI Problems}} \& {{Perspectives}}}, volume~4,
  pages 369--377. {Springer Berlin Heidelberg}, {Berlin, Heidelberg}, 2011.

\bibitem{efronLeastAngleRegression2004}
Bradley Efron, Trevor Hastie, Iain Johnstone, and Robert Tibshirani.
\newblock Least angle regression.
\newblock {\em Ann. Statist.}, 32(2), April 2004.

\bibitem{eftangHpCertifiedReduced2011}
Jens~L. Eftang, David~J. Knezevic, and Anthony~T. Patera.
\newblock An {\emph{hp}} certified reduced basis method for parametrized
  parabolic partial differential equations.
\newblock {\em Mathematical and Computer Modelling of Dynamical Systems},
  17(4):395--422, August 2011.

\bibitem{eftangHpCertifiedReduced2010}
Jens~L. Eftang, Anthony~T. Patera, and Einar~M. R{\o}nquist.
\newblock An "\$hp\$" {{Certified Reduced Basis Method}} for {{Parametrized
  Elliptic Partial Differential Equations}}.
\newblock {\em SIAM J. Sci. Comput.}, 32(6):3170--3200, January 2010.

\bibitem{foucartMathematicalIntroductionCompressive2013}
Simon Foucart and Holger Rauhut.
\newblock {\em A {{Mathematical Introduction}} to {{Compressive Sensing}}}.
\newblock Applied and {{Numerical Harmonic Analysis}}. {Springer New York},
  {New York, NY}, 2013.

\bibitem{giraldiWeaklyIntrusiveLowRank2019}
Loic Giraldi and Anthony Nouy.
\newblock Weakly {{Intrusive Low-Rank Approximation Method}} for {{Nonlinear
  Parameter-Dependent Equations}}.
\newblock {\em SIAM J. Sci. Comput.}, 41(3):A1777--A1792, January 2019.

\bibitem{haasdonkTrainingSetMultiple2011a}
Bernard Haasdonk, Markus Dihlmann, and Mario Ohlberger.
\newblock A training set and multiple bases generation approach for
  parameterized model reduction based on adaptive grids in parameter space.
\newblock {\em Mathematical and Computer Modelling of Dynamical Systems},
  17(4):423--442, August 2011.

\bibitem{herzetPerformanceGuaranteesVariational2020}
C.~Herzet and M.~Diallo.
\newblock Performance guarantees for a variational ``multi-space'' decoder.
\newblock {\em Adv Comput Math}, 46(1):10, February 2020.

\bibitem{kaiserSparseIdentificationNonlinear2018}
E.~Kaiser, J.~N. Kutz, and S.~L. Brunton.
\newblock Sparse identification of nonlinear dynamics for model predictive
  control in the low-data limit.
\newblock {\em Proc. R. Soc. A.}, 474(2219):20180335, November 2018.

\bibitem{kaulmannOnlineGreedyReduced2013}
Sven Kaulmann and Bernard Haasdonk.
\newblock Online greedy reduced basis construction using dictionaries.
\newblock In {\em VI International Conference on Adaptive Modeling and
  Simulation ADMOS 2013}, page~12, March 2013.

\bibitem{loggDOLFINAutomatedFinite2010}
Anders Logg and Garth~N. Wells.
\newblock {{DOLFIN}}: {{Automated}} finite element computing.
\newblock {\em ACM Trans. Math. Softw.}, 37(2):1--28, April 2010.

\bibitem{madayGeneralMultipurposeInterpolation2009}
Yvon Maday, Ngoc Cuong~Nguyen, Anthony T.~Patera, and S.~H.~Pau.
\newblock A general multipurpose interpolation procedure: The magic points.
\newblock {\em Communications on Pure \& Applied Analysis}, 8(1):383--404,
  2009.

\bibitem{madayParameterizedbackgroundDataweakApproach2015}
Yvon Maday, Anthony~T. Patera, James~D. Penn, and Masayuki Yano.
\newblock A parameterized-background data-weak approach to variational data
  assimilation: Formulation, analysis, and application to acoustics.
\newblock {\em Int. J. Numer. Meth. Engng}, 102(5):933--965, May 2015.

\bibitem{madayPBDWStateEstimation2015}
Yvon Maday, Anthony T, James~D Penn, and Masayuki Yano.
\newblock {{PBDW State Estimation}}: {{Noisy Observations}};
  {{Configuration-Adaptive Background Spaces}}; {{Physical Interpretations}}.
\newblock {\em ESAIM: Proc.}, 50:144--168, March 2015.

\bibitem{mairalOnlineDictionaryLearning2009}
Julien Mairal, Francis Bach, Jean Ponce, and Guillermo Sapiro.
\newblock Online dictionary learning for sparse coding.
\newblock In {\em Proceedings of the 26th Annual International Conference on
  Machine Learning}, pages 689--696, {Montreal Quebec Canada}, June 2009.
  {ACM}.

\bibitem{milkPyMORGenericAlgorithms2016}
Ren{\'e} Milk, Stephan Rave, and Felix Schindler.
\newblock {{pyMOR}} -- {{Generic Algorithms}} and {{Interfaces}} for {{Model
  Order Reduction}}.
\newblock {\em SIAM J. Sci. Comput.}, 38(5):S194--S216, January 2016.

\bibitem{nicholsCoarseReducedModel2022}
James Nichols.
\newblock Coarse reduced model selection for nonlinear state estimation.
\newblock {\em ANZIAMJ}, 62:C192--C207, February 2022.

\bibitem{peherstorferLocalizedDiscreteEmpirical2014}
Benjamin Peherstorfer, Daniel Butnaru, Karen Willcox, and Hans-Joachim
  Bungartz.
\newblock Localized {{Discrete Empirical Interpolation Method}}.
\newblock {\em SIAM J. Sci. Comput.}, 36(1):A168--A192, January 2014.

\bibitem{taddeiAdaptiveParametrizedBackgroundDataWeak2017}
Tommaso Taddei.
\newblock An {{Adaptive Parametrized-Background Data-Weak}} approach to
  variational data assimilation.
\newblock {\em ESAIM: M2AN}, 51(5):1827--1858, September 2017.

\bibitem{temlyakovNonlinearKolmogorovWidths1998}
V.~N. Temlyakov.
\newblock Nonlinear {{Kolmogorov}} widths.
\newblock {\em Math Notes}, 63(6):785--795, June 1998.

\bibitem{tibshiraniLassoProblemUniqueness2013}
Ryan~J. Tibshirani.
\newblock The lasso problem and uniqueness.
\newblock {\em Electron. J. Statist.}, 7(none), January 2013.

\bibitem{troppImprovedAnalysisSubsampled2011}
Joel~A. Tropp.
\newblock Improved analysis of the subsampled randomized {{Hadamard}}
  transform.
\newblock {\em Adv. Adapt. Data Anal.}, 03(01n02):115--126, April 2011.

\bibitem{washabaughNonlinearModelReduction2012}
Kyle Washabaugh, David Amsallem, Matthew Zahr, and Charbel Farhat.
\newblock Nonlinear {{Model Reduction}} for {{CFD Problems Using Local
  Reduced-Order Bases}}.
\newblock In {\em 42nd {{AIAA Fluid Dynamics Conference}} and {{Exhibit}}},
  {New Orleans, Louisiana}, June 2012. {American Institute of Aeronautics and
  Astronautics}.

\bibitem{woodruffComputationalAdvertisingTechniques2014}
David~P. Woodruff.
\newblock Computational {{Advertising}}: {{Techniques}} for {{Targeting
  Relevant Ads}}.
\newblock {\em FNT in Theoretical Computer Science}, 10(1-2):1--157, 2014.

\bibitem{zahmInterpolationInverseOperators2016}
Olivier Zahm and Anthony Nouy.
\newblock Interpolation of {{Inverse Operators}} for {{Preconditioning
  Parameter-Dependent Equations}}.
\newblock {\em SIAM J. Sci. Comput.}, 38(2):A1044--A1074, January 2016.

\end{thebibliography}

\end{document}